\documentclass[10pt]{article}         
\usepackage{amsmath,amsfonts,amssymb,mathrsfs,amsthm,tikz-cd,tikz,array,quiver}
\usetikzlibrary{shapes.geometric, arrows}
\tikzstyle{startstop} = [rectangle, rounded corners, minimum width=3cm, minimum height=1cm,text centered, draw=black, fill=red!0]
\tikzstyle{io} = [trapezium, trapezium left angle=70, trapezium right angle=110, minimum width=3cm, minimum height=1cm, text centered, draw=black, fill=blue!0]
\tikzstyle{process} = [rectangle, minimum width=3cm, minimum height=1cm, text centered, draw=black, fill=orange!0]
\tikzstyle{decision} = [diamond, minimum width=3cm, minimum height=1cm, text centered, draw=black, fill=green!0]
\tikzstyle{arrow} = [thick,->,>=stealth]  
\usepackage[utf8]{inputenc}
\usepackage{hyperref}
\usepackage{cite}

\newtheorem{thm}{Theorem}[section]

\theoremstyle{definition}
\newtheorem{defn}[thm]{Definition}
\newtheorem{Thm}[thm]{Theorem}
\newtheorem{THM}{Theorem}
\newtheorem{prop}[thm]{Proposition}
\newtheorem{lem}[thm]{Lemma}
\newtheorem{cor}[thm]{Corollary}

\theoremstyle{remark}
\newtheorem{rem}[thm]{Remark}

\title{The $m$-step Solvable Mono-anabelian Geometry of Number Fields}
\author{
  Yu Mao\\
  \texttt{ym382@exeter.ac.uk}
  \and
  Mohamed Sa\"idi\\
  \texttt{m.saidi@exeter.ac.uk}
}

\date{}
\begin{document}
\maketitle
\begin{abstract}
The goal of this paper is to develop a group-theoretic algorithm, to reconstruct a number field (together with its maximal $m$-step solvable extension for some positive integer $m \geq 3$) from the maximal $m+9$-step solvable quotient of its absolute Galois group. If $K$ is an imaginary quadratic field or $\mathbb{Q}$, we establish a group-theoretic reconstruction algorithm of $K$ from the maximal $6$-step solvable quotient of its absolute Galois group.
\end{abstract}
\newpage
\tableofcontents
\newpage
\section*{Acknowledgement}
The authors would like to thank Yuichiro Hoshi and Akio Tamagawa for valuable discussions.
\section*{Conventions}
Throughout this paper, we use the following conventions. \par 
\smallskip
\textbf{Sets}
\begin{itemize}
\item We write $\mathbb{N}$, $\mathbb{Z}$, $\mathbb{Q}$, $\mathbb{R}$, $\mathbb{C}$, $\mathbb{Z}_p$, $\mathbb{Q}_p$ (for some prime number $p$) for the set of natural numbers, integers, rational numbers, real numbers, complex numbers, $p$-adic integers and $p$-adic numbers respectively. We will clarify when we use additional structures (e.g. group structure, ring structure, etc.) on the above sets if necessary.
\item We write $\mathbb{Z}_{\geq r}$ for the set of integers greater than or equal to an integer $r$. 
\end{itemize}
\textbf{Monoids}
\begin{itemize}
\item Let $M$ be a commutative multiplicative monoid. We write 
$$
M_{\times} := M \cup \{0_M\}
$$ 
for a commutative monoid containing $M$ by setting $0_M \times a := 0_M$ for all $a \in M$ and $0_M \times 0_M := 0_M$.
\end{itemize}
\textbf{Groups}
\begin{itemize}
\item Let $G$ be a profinite group, we write $G^{\text{ab}}$ for the (topological) abelianisation of $G$. Thus, $G^{\text{ab}} := G/\overline{[G,G]}$, where $\overline{[G,G]}$ denotes the closure of the commutator subgroup of $G$.
\item All homomorphisms of profinite groups are assumed to be continuous.
\item Let $G$ be an abelian group, we write $G_{\text{tor}}$ for the torsion subgroup of $G$. If $G$ is an abelian profinite group, then $G_{\text{tor}}$ is defined to be the closure of the subgroup consisting of all torsion elements. We write $G^{/\text{tor}}$ for the quotient $G/G_{\text{tor}}$.
\item Let $G$ be a profinite group, we write $G^{\text{ab/tor}}$ for the quotient $G^{\text{ab}}/(G^{\text{ab}})_{\text{tor}}$.
\item Let $G$ be a group, we write $\widehat{G}$ for the profinite completion of $G$. Thus,
$$
\widehat{G} := \varprojlim_{N} G/N
$$
where $N$ ranges over all normal subgroup of $G$ of finite index. Moreover, if $G$ is an abelian (topological) group, we define
$$
G^{\wedge} := \varprojlim_{n \geq 1} G/nG.
$$
Notice that $\widehat{G} = G^{\wedge}$ if $G$ is (topologically) finitely generated (in this case the subgroups $\{nG\}_{n \geq 1}$ form a cofinal system of subgroups of finite index).
\item Let $G$ be a profinite group. We define the closed subgroups $G^{[i]}$ of $G$ as follows:  $G^{[0]} := G$ and $G^{[i]} := \overline{[G^{[i-1]},G^{[i-1]}]}$ for $i \geq 1$ ($G^{[i]}$ is known as the $i$-th derived subgroup of $G$). We define $G^i := G/G^{[i]}$ the maximal $i$-step solvable quotient of $G$. By definition we have $G^1 = G^{\text{ab}}$. Moreover, for all $1 \leq i \leq j$, we write 
$$
G^{[j,i]} := \text{ker}(G^j \twoheadrightarrow G^i) = (G^j)^{[i]} = (G^{[i]})^{j-i}.
$$
\item Let $G$ be an abelian group and $n\geq 1$ an integer, we write $G[n]$ for the subgroup of $G$ consisting of all $n$-torsion elements for $n \geq 1$, i.e. $G[n] = \text{ker}(G \xrightarrow{\times n} G)$.
\item Let $G$ be a profinite group, and let $p$ be a prime number. We write $G^{(p)}$ for the maximal pro-$p$ quotient of $G$. Moreover, we write $G^{(p')}$ for the maximal prime-to-$p$ quotient of $G$. Thus, $G^{(p')} := \varprojlim_N G/N$, where $N$ ranges over all open normal subgroups of $G$ such that $[G:N]$ is not divisible by $p$.
\item Let $G$ be a profinite group, we write $G^{\text{sol}}$ for the maximal prosolvable quotient of $G$.
\end{itemize}
\textbf{Rings and Fields}
\begin{itemize}
\item Let $R$ be a ring (we always assume that $R$ is commutative with $1 \neq 0$), we write $R^{\rhd}$ for the multiplicative monoid $R \setminus \{0\}$. Moreover, we write $R_+$ for the underlying additive group of $R$.
\item Let $K$ be a field. Fix an algebraic closure $\overline{K}$ of $K$. We write $K^{\text{sep}},K^{\text{sol}}$ and $K^{\text{ab}}$ for the separable closure, the maximal prosolvable extension and the maximal abelian extension of $K$ contained in $\overline{K}$, respectively.
\item We write $G_K := \text{Gal}(K^{\text{sep}}/K)$ for the absolute Galois group of $K$. We write $G_K^{\text{sol}},G_K^{m}$ and $G_K^{\text{ab}}$ for the maximal prosolvable, maximal $m$-step solvable and the maximal abelian quotient of $G_K$, respectively. Notice that $G_K^{\text{ab}} = G_K^1$.
\item We write $K_m$ for the maximal $m$-step solvable extension of $K$. Thus, $K_m := (K^{\text{sep}})^{G_K^{[m]}}$.
\item Let $K$ be a field. We write $\text{char}(K)$ for the characteristic of $K$.
\item Let $K$ be a field which contains the $n$-th roots of unity for all $n \geq 1$, we write $\mu(K) := (K^{\times})_{\text{tor}}$. Hence
$$
\mu(K) \cong \begin{cases} \mathbb{Q}/\mathbb{Z}, ~\text{if}~\text{char}(K)=0 \\ \bigoplus_{\ell \neq \text{char}(K)} \mathbb{Q}_{\ell}/\mathbb{Z}_{\ell}, ~\text{if}~\text{char}(K)>0. \end{cases}
$$
Moreover, we write $\mu_n(K) := \mu(K)[n]$, and $\Lambda(K) := \varprojlim_n \mu_n(K)$. We shall refer to $\Lambda(K)$ as the cyclotome associated to $K$. Hence
$$
\Lambda(K) \cong \begin{cases} \widehat{\mathbb{Z}}, ~\text{if}~\text{char}(K)=0 \\ \widehat{\mathbb{Z}}^{(\text{char}(K)')} := \prod_{\ell \neq \text{char}(K)} \mathbb{Z}_{\ell}, ~\text{if}~\text{char}(K)>0. \end{cases}
$$
Notice that the above isomorphisms are non-canonical.
\item Let $K$ be a number field, a prime $\mathfrak{p}$ of $K$ corresponds to an equivalence class of valuations $\text{ord}_{\mathfrak{p}}$ of $K$. The prime $\mathfrak{p}$ is said to be non-archimedean if the corresponding valuation $\text{ord}_{\mathfrak{p}}$ is non-archimedean, otherwise $\mathfrak{p}$ is said to be archimedean. We write $\mathscr{P}_K$ for the set of all primes of $K$, $\mathscr{P}_K^{\text{fin}}$ for the set of all non-archimedean primes of $K$, and $\mathscr{P}_K^{\text{inf}}$ for the set of all archimedean primes of $K$. Notice that $\mathscr{P}_K^{\text{inf}}$ is always finite.
\end{itemize}
\section{Introduction}
Anabelian geometry is roughly speaking the study of geometric objects in terms of their \'etale fundamental groups. The simplest geometric object is $\text{Spec}(K)$ for some field $K$. In this case, the \'etale fundamental group of $\text{Spec}(K)$ is isomorphic to $G_K$, and it is natural to ask the following question: how much information on the field $K$ is determined by $G_K$?\par 
Clearly this question depends on the field $K$, as if $K$ is a finite field for example, $K$ has absolute Galois group isomorphic to $\widehat{\mathbb{Z}}$, and hence the isomorphy type of $K$ cannot be determined by $G_K$. Also, if $K$ is algebraically closed or real closed, then $G_K$ is trivial or isomorphic to $\mathbb{Z}/2\mathbb{Z}$, in this case also the isomorphy type of $K$ is not determined by $G_K$ (nevertheless, if $G_K \cong \mathbb{Z}/2\mathbb{Z}$, we can conclude that $\text{char}(K) = 0$ by the Artin-Schreier Theorem). \par 
When $K$ is a $p$-adic local field for some prime number $p$, the structure of $G_K$ is complicated enough, but there are still non-isomorphic $p$-adic local fields $K$ and $L$ with isomorphic absolute Galois groups, e.g. see p.9 Theorem in \cite{JR}. However, if $G_K$ is equipped with the ramification filtration in upper numbering (e.g. see Definition 2.3 in \cite{Mzk1}), then the isomorphy type of $K$ is determined by the isomorphy type of $G_K$ equipped with the ramification filtration, see Theorem 4.2 in \cite{Mzk1}. \par 
When $K$ is a number field, its absolute Galois group is complicated enough to determine the isomorphy type of $K$. More precisely, we have the following
\begin{Thm}[The Neukirch-Uchida Theorem, 1970's]
Let $K,L$ be number fields, and $\sigma: G_L \xrightarrow{\sim} G_K$ an isomorphism. Then $\sigma$ induces a unique isomorphism of fields $\tau: K^{\text{sep}} \xrightarrow{\sim} L^{\text{sep}}$ satisfying $\sigma(g) = \tau^{-1} g \tau$ for all $g \in G_L$, which restricts to an isomorphism of number fields $\tau|_K : K \xrightarrow{\sim} L$. 
\end{Thm}
However, the structure of $G_K$ is unknown for any number field $K$, so Theorem 1.1 does not have explicit applications a priori. Considering smaller quotients of $G_K$ is desirable. Later in 1976, Uchida proved a stronger version of Theorem 1.1 by considering the corresponding maximal pro-solvable quotient instead of the full absolute Galois group, see \cite{Uchida1}. In 2019, Sa\"idi and Tamagawa proved a greatly improved version of Theorem 1.1 by only considering a solvable quotient with a restricted length of solvability:
\begin{Thm}[Sa\"idi-Tamagawa, 2019, \cite{ST1}]
Let $K,L$ be number fields and $\sigma_{m+3}: G_L^{m+3} \xrightarrow{\sim} G_K^{m+3}$ an isomorphism, which induces an isomorphism $\sigma_m: G_L^m \xrightarrow{\sim} G_K^m$. Then the following assertions hold:\par 
(i) If $m \geq 0$, there exists a field isomorphism $\tau_m : K_m \xrightarrow{\sim} L_m$ such that $\sigma_m(g) = \tau_m^{-1} g \tau_m$ for all $g \in G_L^m$, and $\tau_m$ restricts to a field isomorphism $K \xrightarrow{\sim} L$. \par 
(ii) If $m \geq 2$ (resp. $m=1$), then the isomorphism $\tau_m : K_m \xrightarrow{\sim} L_m$ (resp. $\tau: K \xrightarrow{\sim} L$ induced by $\tau_1: K_1 \xrightarrow{\sim} L_1$) in $(i)$ is uniquely determined by the condition $\sigma_m(g) = \tau_m^{-1} g \tau_m$ for all $g \in G_L^m$.
\end{Thm}
However, both Theorem 1.1 and Theorem 1.2 and their proofs are not explicit constructions, i.e. the two theorems start from isomorphisms of Galois groups and obtain isomorphisms of fields, which includes no explicit reconstruction of the field structures. These kind of results are called \textbf{bi-anabelian results}. On the other hand, explicit reconstruction results of number fields from their (various quotients of) absolute Galois groups are known as \textbf{mono-anabelian results}. \par
For number fields, we have the following mono-anabelian result by Hoshi:
\begin{Thm}[Hoshi, 2021, \cite{Ho2}]
Let $G$ be a profinite group which is isomorphic to $G_K^{\text{sol}}$ for some number field $K$. Then one can group-theoretically reconstruct a solvably closed field (i.e. a field with no non-trivial abelian extension) $\widetilde{F}(G)$ with the action of $G$ and the fixed subfield $F(G) := \widetilde{F}(G)^G$ which is a number field, such that:\par 
(1) $\widetilde{F}(G)$ and $F(G)$ fit into the following commutative diagram:
$$
\begin{tikzcd}
\widetilde{F}(G) \arrow[r,"\sim"] & K^{\text{sol}} \\
F(G) \arrow[u,hookrightarrow] \arrow[r,"\sim"] & K \arrow[u,hookrightarrow]
\end{tikzcd}
$$
where the horizontal arrows are Galois-equivariant (with respect to $G \xrightarrow{\sim} G_K^{\text{sol}}$) isomorphisms and the vertical arrows are natural inclusions. \par
(2) We have a group isomorphism $\text{Gal}(\widetilde{F}(G)/F(G)) \xrightarrow{\sim} G_K^{\text{sol}}$.
\end{Thm}
The goal of this paper is to prove an analogue of Theorem 1.3 in the context of Theorem 1.2, more precisely, we prove the following:
\begin{THM}
Let $m \geq 3$, and let $G$ be a profinite group which is isomorphic to $G_K^{m+9}$ for some number field $K$. Then one can group-theoretically reconstruct a field $F_m(G^m)$ with the action of $G^m$ and the fixed subfield $F(G^m) := F_m(G^m)^{G^m}$ which is a number field, such that: \par
(1) We have a group isomorphism $\text{Gal}(F_m(G^m)/F(G)) \xrightarrow{\sim} G_K^{m}$. \par
(2) $F_m(G^m)$ and $F(G^m)$ fit into the following commutative diagram:
$$
\begin{tikzcd}
F_m(G^m) \arrow[r,"\sim"] & K_m \\
F(G^m) \arrow[u,hookrightarrow] \arrow[r,"\sim"] & K \arrow[u,hookrightarrow]
\end{tikzcd}
$$
where the horizontal arrows are Galois-equivariant isomorphisms (with respect to $G^m\xrightarrow{\sim}G_K^m$ induced by $G \xrightarrow{\sim} G_K^{m+9}$) and the vertical arrows are natural inclusions. Moreover, the isomorphisms in the diagram are functorial with respective to the isomorphy type of $G^m$.\par 
\end{THM}
If we only consider imaginary quadratic fields and $\mathbb{Q}$, then we have the following result.
\begin{THM}
If $G$ is a profinite group which is isomorphic to $G_K^6$, where $K$ is an imaginary quadratic field or $\mathbb{Q}$, then one can group-theoretically reconstruct a field $F(G)$ such that $F(G) \xrightarrow{\sim} K$.
\end{THM} 
The proofs of Theorem 1 and Theorem 2 are fundamentally different. One of the key differences is that the proof of Theorem 1 uses the statement of Theorem 1.2 in an essential way, but the proof of Theorem 2 does not involve the statement of Theorem 1.2. From this point of view, one can conclude that Theorem 1 does not provide an alternative proof to Theorem 1.2. On the other hand, Theorem 2 does provide an alternative proof to a weaker version of Theorem 1.2, but only for imaginary quadratic fields and $\mathbb{Q}$. More precisely, if $K,L$ are both imaginary quadratic or $\mathbb Q$, then Theorem 2 implies that if $G_K^6 \cong G_L^6$ then $K \cong L$.
\section{Reconstruction of Local Invariants}
In this section, we use the local theory established by Sa\"idi and Tamagawa in \cite{ST1} in order to reconstruct various local invariants group theoretically.
\begin{defn}
Let $G$ be a profinite group, and let $m \geq 0$ be an integer. We say that $G$ is of $\text{GSC}^m$-type (c.f. Definition 3.2 in \cite{Ho1} for an analogous definition of profinite groups of $\text{GSC}$-type) if $G$ admits the following data \par 
(1) A number field $K$. \par 
(2) A maximal $m$-step solvable extension $K_m$ of $K$. \par 
(3) A continuous isomorphism $\alpha_m: \text{Gal}(K_m/K) \xrightarrow{\sim} G$.
\end{defn}
\begin{rem}
(i) Let $m' \geq m \geq 0$. Let $G$ be a profinite group of $\text{GSC}^{m'}$-type. Then $\widetilde{G} := G^m$ is of $\text{GSC}^m$-type. In this case we say that $\widetilde{G}$ arises from a profinite group of $\text{GSC}^{m'}$-type. \par 
(ii) Let $m,n \geq 0$. Let $G$ be a profinite group of $\text{GSC}^{m+n}$-type. Let $H \subset G^m$ be an open subgroup, we write $\widetilde{H}$ for the inverse image of $H$ in $G$. Then $\widetilde{H}^n$ is a profinite group of $\text{GSC}^n$-type.
\end{rem}
In particular, we view a profinite group of $\text{GSC}^m$-type as an \textbf{abstract profinite group} rather than a distinguished quotient of the absolute Galois group of some number field.
\begin{defn}
Let $m \geq 1$ be an integer. Let $G$ be a profinite group of $\text{GSC}^{m+2}$-type. Then we write
$$
\widetilde{\text{Dec}}(G^m)
$$
for the image of $\text{St}(\mathcal{D}_{m+1})$ (see Proposition 1.24 in \cite{ST1} for the definition of $\text{St}(\mathcal{D}_{m+1})$) via $\alpha_m: G_K^m \xrightarrow{\sim} G^m$ (it follows from Theorem 1.25 in \cite{ST1} that this image is independent from the choice of isomorphism $\alpha_m$). Moreover, $G^m$ acts on $\widetilde{\text{Dec}}(G^m)$ by conjugation (c.f. Loc cit), we write $\text{Dec}(G^1) := \widetilde{\text{Dec}}(G^m)/G^m$.
\end{defn}
\begin{rem}
Essentially, $\widetilde{\text{Dec}}(G^m)$ corresponds to the set of decomposition subgroups of $G_K^m$ at non-archimedean primes of $K_m$, and $\text{Dec}(G^1)$ is in bijection with the set of non-archimedean primes of $K$. A detailed proof can be found in Section 1.3 in \cite{ST1}. One should notice that $\widetilde{\text{Dec}}(G^m)$ depends on $m$ (c.f. Definition 2.1 and Definition 2.2), but $\text{Dec}(G^1)$ in the above definition does not depend on $m$. For example, if $G$ is a profinite group of $\text{GSC}^m$-type for some $m \geq 3$, then $\widetilde{\text{Dec}}(G)$ and $\widetilde{\text{Dec}}(G^{m-1})$ are different, but modulo the action of $G$ (resp. $G^{m-1}$) by conjugation on both sets we have the same set $\text{Dec}(G^1)$, i.e. $\widetilde{\text{Dec}}(G)/G = \widetilde{\text{Dec}}(G^{m-1})/G^{m-1}$.
\end{rem}
\begin{defn}
Let $G$ be a profinite group of $\text{GSC}^{m+7}$-type for some integer $m \geq 0$. Let $H \subset G^{m+3}$ be an open subgroup. We write $\widetilde{H}$ for the inverse image of $H$ in $G$. Thus $\widetilde{H}^4$ can be identified with $G_L^4$ for some finite intermediate subextension $L/K$ of $K_{m+3}/K$. We set (by abusing notions)
$$
\widetilde{\text{Dec}}(H) := \widetilde{\text{Dec}}(\widetilde{H}^2).
$$
Moreover, we write $\text{Dec}(H) := \widetilde{\text{Dec}}(H)/\widetilde{H}^2$.
\end{defn}
\begin{Thm}
Let $G$ be a profinite group of $\text{GSC}^6$-type. Let $D \in \widetilde{\text{Dec}}(G^4)$. It follows from Proposition 1.1 (i) in \cite{ST1} that $D \xrightarrow{\sim} D_{\mathfrak{p}_4} = \text{Gal}((K_{\mathfrak{p}})_4/K_{\mathfrak{p}})$ for some uniquely determined $\mathfrak{p}_4 \in \mathscr{P}_{K_4}^{\text{fin}}$ and $\mathfrak{p}$ is the image of $\mathfrak{p}_4$ in $K$. Then we can group-theoretically reconstruct the following objects from $D$: \par 
(i) The residue characteristic $p_\mathfrak{p}$ of $K_{\mathfrak{p}}$. \par 
(ii) The inertia degree $f_\mathfrak{p}$ of $K_{\mathfrak{p}}$. \par
(iii) The degree $d_{K_{\mathfrak{p}}} := [K_{\mathfrak{p}}:\mathbb{Q}_{p_{\mathfrak{p}}}].$ \par 
(iv) The absolute ramification index $e_{\mathfrak{p}}$ of $K_{\mathfrak{p}}$. \par 
(v) The inertia subgroup $I_{\mathfrak{p}_2}$ of $D_{\mathfrak{p}_2}$, where $\mathfrak{p}_2$ is the image of $\mathfrak{p}_4$ in $K_2$. \par 
(vi) The wild inertia subgroup $W_{\mathfrak{p}_2}$ of $I_{\mathfrak{p}_2}$. \par 
(vii) The Frobenius element $\text{Frob}_{\mathfrak{p}}$, i.e. the topological generator of $D_{\mathfrak{p}_2}/I_{\mathfrak{p}_2}$ corresponding to the Frobenius automorphism. \par 
(viii) The multiplicative group $K_{\mathfrak{p}}^{\times}$. \par 
(ix) The group of roots of unity $\mu(K_{\mathfrak{p}}^{\text{ab}})$ contained in $K_{\mathfrak{p}}^{\text{ab}}$ and the cyclotome $\Lambda(K_{\mathfrak{p}}^{\text{ab}})$ associated to $K_{\mathfrak{p}}^{\text{ab}}$ as a $D$-module (whose action factors through $D^1$). \par 
While the objects in assertions (i)-(iv) can be reconstructed from $D^1$, the objects in assertions (v)-(vii) can be reconstructed from $D^3$, and the objects in assertion (ix) can be reconstructed from $D^4$.
\end{Thm}
\begin{proof}
Assertions (i)-(iv) are well-known from local class field theory and the explicit structure of $K_{\mathfrak{p}}^{\times}$ (e.g. see Proposition 5.7 in \cite{ANT}). In particular, the residue characteristic of $K_{\mathfrak{p}}$ can be characterised as the unique prime number $l$ such that $\log_{l}|D^{\text{ab/tor}}/lD^{\text{ab/tor}}| \geq 2$. We write $p_D$ for this prime number. The inertia degree of $K_{\mathfrak{p}}$ can be characterised as $\log_{p_D}|1+(D^{\text{ab}})_{\text{tor}}^{(p_D')}|$, we write $f_D$ for this integer. The degree $d_{K_{\mathfrak{p}}} = [K_{\mathfrak{p}}: \mathbb{Q}_{p_D}]$ can be characterised as $\log_{p_D}|D^{\text{ab/tor}}/pD^{\text{ab/tor}}| - 1$, we write $d_D$ for this integer. The absolute ramification index $e_{\mathfrak{p}}$ can be characterised as $d_D/f_D$, we write $e_D$ for this integer. \par 
For assertion (v)-(vii), let $H \subset D^2$ be an open subgroup and write $\widetilde{H}$ for the inverse image of $H$ in $D^3$. Then it holds that $\widetilde{H}^1$ is isomorphic to $G_L^{\text{ab}}$ for some finite metabelian extension $L$ of $K_{\mathfrak{p}}$. We apply (i)-(iv) to $\widetilde{H}^1$, and set $e_H := e_{\widetilde{H}^1}$. Then the inertia subgroup of $I_{\mathfrak{p}_2}$ can be characterised as the intersection 
$$
\bigcap_H H
$$
where $H$ ranges over all open normal subgroups of $D^2$ such that $e_{\widetilde{H}^1} = e_D$. We write $I(D^2)$ for this group. Moreover, $W_{\mathfrak{p}_2}$ can be characterised as the unique pro-$p_D$-Sylow subgroup of $I(D^2)$. We write $W(D^2)$ for this group. The Frobenius element $\text{Frob}_{\mathfrak{p}}$ can be characterised as the unique element of $D^{\text{unr}} := D^2/I(D^2)$ whose action on $I(D^2)/W(D^2) \cong \widehat{\mathbb{Z}}^{(p_D')}(1)$ is given by multiplication by $p_D^{f_D}$ (see the discussion after Proposition 7.5.1 in \cite{NSW}). We write $\text{Frob}(D)$ for this element. \par 
For assertion (viii), consider the following commutative diagram with exact rows (arising from the local class field theory for $K_{\mathfrak{p}}$):
$$
\begin{tikzcd}
	1 & {\mathcal{O}_\mathfrak{p}^{\times}} & {K_{\mathfrak{p}}^{\times}} & {\mathbb{Z}} & 1 \\
	1 & {\mathcal{O}_{\mathfrak{p}}^{\times}} & {\widehat{K_{\mathfrak{p}}^{\times}}} & {\widehat{\mathbb{Z}}} & 1 \\
	1 & {I_{\mathfrak{p}_1}} & {D_{\mathfrak{p}_1}} & {D_{\mathfrak{p}}^{\text{unr}}} & 1
	\arrow[from=1-1, to=1-2]
	\arrow[from=1-2, to=1-3]
	\arrow[no head, from=1-2, to=2-2]
	\arrow[shift right, no head, from=1-2, to=2-2]
	\arrow["{\text{ord}_{\mathfrak{p}}}", from=1-3, to=1-4]
	\arrow[hook, from=1-3, to=2-3]
	\arrow[from=1-4, to=1-5]
	\arrow[hook, from=1-4, to=2-4]
	\arrow[from=2-1, to=2-2]
	\arrow[from=2-2, to=2-3]
	\arrow["\wr", from=2-2, to=3-2]
	\arrow[from=2-3, to=2-4]
	\arrow["{\theta_{\mathfrak{p}}}", from=2-3, to=3-3]
	\arrow[from=2-4, to=2-5]
	\arrow["\wr", from=2-4, to=3-4]
	\arrow[from=3-1, to=3-2]
	\arrow[from=3-2, to=3-3]
	\arrow[from=3-3, to=3-4]
	\arrow[from=3-4, to=3-5]
\end{tikzcd}
$$
where $\text{ord}_{\mathfrak{p}}$ is the valuation on $K_{\mathfrak{p}}^{\times}$, and $\theta_{\mathfrak{p}}$ is the local reciprocity isomorphism. In particular, we write $\text{Frob}(D)^{\mathbb{Z}}$ for the discrete infinite cyclic subgroup of $D^{\text{unr}}$ generated by $\text{Frob}(D)$, hence $\text{Frob}(D)^{\mathbb{Z}}$ coincide with the image of $\mathbb{Z}$ through the composite
$$
\mathbb{Z} \hookrightarrow \widehat{\mathbb{Z}} \xrightarrow{\sim} D_{\mathfrak{p}}^{\text{unr}} \xrightarrow{\sim} D^{\text{unr}}
$$
where the first two maps are obtained from the diagram above, and the final isomorphism arises from the definition of $D^{\text{unr}}$ defined above. The multiplicative group $K_{\mathfrak{p}}^{\times}$ can be characterised as the fibre product
$$
D^1 \times_{D^{\text{unr}}} \text{Frob}(D)^{\mathbb{Z}}.
$$
We write $k^{\times}(D^1)$ for this group. Moreover, let $H \subset D^1$ be an open subgroup, we write $\widetilde{H}$ for the inverse image of $H$ in $D$. Hence $\widetilde{H}^3$ can be identified with $G_F^3$ for some finite abelian extension $F/K_{\mathfrak{p}}$. We can apply (i) - (viii) to $\widetilde{H}^3$, and obtain 
$$
k^{\times}(\widetilde{H}^1) := \widetilde{ H}^1 \times_{\widetilde{H}^{\text{unr}}} \text{Frob}(\widetilde{H}^3)^{\mathbb{Z}}.
$$
The natural inclusion $\widetilde{H} \hookrightarrow D$ determines an injective map
$$
k^{\times}(D^1) \hookrightarrow k^{\times}(\widetilde{H}^1)
$$
induced by the transfer map $\text{Ver}: D^1 \to \widetilde{H}^1$.
By abusing notion, we shall write $k^{\times}(H) := k^{\times}(\widetilde{H}^1)$. Hence we can characterise the group of roots of unity contained in $K_{\mathfrak{p}}^{\text{ab}}$ as:
$$
\mu(D^1) := \varinjlim_H ~k^{\times}(H)_{\text{tor}},
$$
where $H$ ranges over all open subgroups of $D^1$, and the transition maps are the natural inclusions. We write
$$
\Lambda(D^1) := \varprojlim_n ~\mu(D)[n]
$$
for the cyclotome of $D$.
\end{proof}
\begin{rem}
If $k$ is a $p$-adic local field for some prime number $p$. Then we have a similar group-theoretic reconstruction of the following objects starting from $G_k^4$ as in Theorem 2.6: \par 
\begin{itemize}
    \item The residue characteristic $p_k$ of $k$.
    \item The inertia degree $f_k$ of $k$.
    \item The degree $d_k:= [k:\mathbb{Q}_{p_k}]$.
    \item The absolute ramification index $e_k$ of $k$.
    \item The inertia subgroup $I_k^{(2)}$ of $G_k^2$.
    \item The wild inertia subgroup of $I_k^{(2)}$.
    \item The Frobenius element $\text{Frob}_k$, i.e. the topological generator of $G_k^{\text{unr}}$ corresponding to the Frobenuis automorphism.
    \item The multiplicative group $k^{\times}$.
    \item The group of roots of unity $\mu(k^{\text{ab}})$ and the cyclotome $\Lambda(k^{\text{ab}})$, viewed as $G_k^{\text{ab}}$-modules.
\end{itemize}
\end{rem}
\section{Local-Global Synchronisation}
In this section, we prove a result on the \textbf{local-global synchronisation of the cyclotomes}, and use this result in order to construct \textbf{Kummer containers}. \par 
Recall that for a number field $K$, the group of finite ideles can be constructed as follows (e.g. see Chapter VI section 1 \cite{ANT}): If $S \subset \mathscr{P}_K^{\text{fin}}$ is a finite set, then 
$$
\mathbb{I}_S^{\text{fin}} =  \prod_{v \in S} K_v^{\times} \times \prod_{v \notin S}\mathcal{O}_v^{\times}
$$
and 
$$
\mathbb{I}_{K}^{\text{fin}} = \bigcup_{S \subset \mathscr{P}_K^{\text{fin}}} \mathbb{I}_S^{\text{fin}}
$$ where $S$ ranges over all finite subsets of $\mathscr{P}_{K}^{\text{fin}}$. \\
Now let $G$ be a profinite group of $\text{GSC}^5$-type. It follows from Proposition 1.5 in \cite{ST1}, that we have the following commutative diagram
$$
\begin{tikzcd}
\widetilde{\text{Dec}}(G^1) \arrow[r,"\sim"] \arrow[d,"\wr"] & \mathscr{P}_{K_1}^{\text{fin}} \arrow[d,"\wr"] \\ 
\text{Dec}(G^1) \arrow[r,"\sim"] & \mathscr{P}_K^{\text{fin}}
\end{tikzcd}
$$
where the vertical arrows are natural bijections and the horizontal arrows are induced by $\alpha_1$, c.f. Definition 2.1.
\begin{defn}
Let $G$ be a profinite group of $\text{GSC}^5$-type, and let $v \in \text{Dec}(G^1)$. Let $D \in \widetilde{\text{Dec}}(G^3)$ whose image in $\text{Dec}(G^1)$ is $v$. We define
\begin{itemize}
    \item $p_v := p_D$.
    \item $d_v := d_D$. 
    \item $e_v := e_D$.
    \item $f_v := f_D$.
    \item $k^{\times}(v) := k^{\times}(D^1)$.
    \item $\mathcal{O}^{\times}(v) := \text{im}(I(D^2) \hookrightarrow D^2 \twoheadrightarrow D^1)$.
\end{itemize}
It follows from Theorem 2.6 that the objects listed above are well-defined and are independent from the choice of $D$ above $v$.
\end{defn}
\begin{defn}
Let $G$ be a profinite group of $\text{GSC}^5$-type and let $S \subset \text{Dec}(G^1)$ be a finite subset. We write
$$
\mathbb{I}_S^{\text{fin}}(G^1) := \prod_{v \in S} k^{\times}(v) \times \prod_{v \notin S} \mathcal{O}^{\times}(v).
$$
Moreover, we write
$$
\mathbb{I}^{\text{fin}}(G^1) := \bigcup_{S} \mathbb{I}_S^{\text{fin}}(G^1)
$$
where the union is taken over all finite subset $S \subset \text{Dec}(G^1)$.
\end{defn}
\begin{defn}
Let $G$ be a profinite group of $\text{GSC}^6$-type and let $H \subset G^1$ be an open subgroup. We write $\widetilde{H}$ for the inverse image of $H$ in $G$, and
$$
\mathbb{I}^{\text{fin}}(H) := \mathbb{I}^{\text{fin}}(\widetilde{H}^1).
$$
\end{defn}
\begin{rem}
The group $\mathbb{I}^{\text{fin}}(H)$ in Definition 3.3 is well-defined since $\widetilde{H}^5$ is of $\text{GSC}^5$-type.
\end{rem}
\begin{defn}
Let $G$ be a profinite group of $\text{GSC}^6$-type. We shall write (c.f. Proposition 3.7 (4) in \cite{Ho1})
$$
\mu(G^1) := \varinjlim_H ~\text{ker}(\mathbb{I}^{\text{fin}}(H) \to \widetilde{H}^1)_{\text{tor}}
$$
where $H$ ranges over all open subgroups of $G^1$ and the transition maps in the direct limit are obtained as follows: the map $\mathbb{I}^{\text{fin}}(H) \to \widetilde{H}^1$ is obtained by considering the natural inclusions
$$
\mathcal{O}^{\times}(v) \hookrightarrow k^{\times}(v) \hookrightarrow D^1
$$
for some $D \in \widetilde{\text{Dec}}(\widetilde{H}^3)$ whose image in $\text{Dec}(H)$ is $v$. In particular, $\mathbb{I}^{\text{fin}}(H) \to \widetilde{H}^1$ is the (restriction of) global Artin reciprocity map. \par 
Let $H' \subset H$ be an open subgroup. We write $\widetilde{H'}$ and $\widetilde{H}$ for the inverse images of $H'$ and $H$ in $G$. The transfer map
$$
\text{Ver}: \widetilde{H}^1 \to \widetilde{H'}^1
$$
induces the following commutative diagram
$$
\begin{tikzcd}
    	{\mathbb{I}^{\text{fin}}(H)} & {\widetilde{H}^{1}} \\
	{\mathbb{I}^{\text{fin}}(H')} & {\widetilde{H'}^{1}}.
	\arrow[from=1-1, to=1-2]
	\arrow[from=1-1, to=2-1]
	\arrow["{\text{Ver}}", from=1-2, to=2-2]
	\arrow[from=2-1, to=2-2]
\end{tikzcd}
$$
In particular, we obtain a natural map
$$
\text{ker}(\mathbb{I}^{\text{fin}}(H) \to \widetilde{H}^1)_{\text{tor}} \to \text{ker}(\mathbb{I}^{\text{fin}}(H') \to \widetilde{H'}^1)_{\text{tor}}
$$
induced by the natural inclusion $H' \hookrightarrow H$, and we shall set these maps as the transition maps in the definition of $\mu(G^1)$. \par
Moreover, we write
$$
\Lambda(G^1) := \varprojlim_n ~\mu(G^1)[n]
$$
where $n$ ranges over all positive integers.
\end{defn}

\begin{defn}
Let $G$ be a profinite group of $\text{GSC}^7$-type. Let $H \subset G^1$ be an open subgroup, and write $\widetilde{H}$ for its inverse image in $G$. Write
$$
\mu(H) := \mu(\widetilde{H}^1)
$$
and
$$
\Lambda(H) := \Lambda(\widetilde{H}^1)
$$
(c.f. Definition 3.5).
\end{defn}
\begin{thm}[Local-Global Synchronisation]
Let $G$ be a profinite group of $\text{GSC}^6$-type, and let $D \in \widetilde{\text{Dec}}(G^4)$. Then there exists $D$-equivariant isomorphisms
$$
\mu(G^1) \xrightarrow{\sim} \mu(D^1), ~\text{and}~\Lambda(G^1) \xrightarrow{\sim} \Lambda(D^1),
$$
where the $D$-action on $\mu(G^1)$ and $\Lambda(G^1)$ is obtained by restricting the cyclotomic character $\chi: G^1 \to \widehat{\mathbb{Z}}^{\times}$ (which can be reconstructed from $G^3$ by Theorem 1.26 in \cite{ST1}) to $D^1$.
\end{thm}
\begin{proof}
Consider the composite
$$
\text{ker}(\mathbb{I}^{\text{fin}}(G^1) \to G^1) \hookrightarrow \mathbb{I}^{\text{fin}}(G^1) \twoheadrightarrow k^{\times}(v) \xrightarrow[]{\sim} k^{\times}(D^1)
$$
where $\mathbb{I}^{\text{fin}}(G^1) \twoheadrightarrow k^{\times}(v)$ is defined to be the component-wise projection. The above composite determines a map
$$
\text{ker}(\mathbb{I}^{\text{fin}}(G^1) \to G^1)_{\text{tor}} \to k^{\times}(D)_{\text{tor}}. 
$$
On the other hand, for each open subgroup $H \subset G^1$, the intersection $H \cap D^1$ is an open subgroup of $D^1$, we write $D_H$ for this group. We write $\widetilde{D}_H$ for the inverse image of $D_H$ in $D$. Hence we can apply Theorem 2.5 to $\widetilde{D}_H^3$ in order to construct $k^{\times}(\widetilde{D}_H)$, and obtain
$$
\varinjlim_H ~ \text{ker}(\mathbb{I}^{\text{fin}}(H) \to \widetilde{H}^{1})_{\text{tor}} \to \varinjlim_H~ k^{\times}(\widetilde{D}_H^1)_{\text{tor}}. 
$$
Here, the inductive limit is taken over all open subgroups of $G^1$, where the left hand side is $\mu(G^1)$ and the right hand side is $\mu(D)$. Hence we obtain a map
$$
t: \mu(G^1) \to \mu(D^1).
$$
We claim that the map $t$ is an isomorphism. The injectivity of $t$ follows immediately from the proof of Lemma 3.6 (iii) in \cite{Ho1}, the surjectivity of $t$ follows from the fact that every injective endomorphism of $\mathbb{Q}/\mathbb{Z}$ is surjective. Moreover, the isomorphism $t$ is equivariant w.r.t the action given by $\chi|_D$ where $\chi: G^1 \to \widehat{\mathbb{Z}}^{\times}$ is the cyclotomic character (which can be reconstructed from $G^3$ by Theorem 1.26 in \cite{ST1}) and $\chi|_D$ is the restriction of $\chi$ to $D^1$. Hence we also have a $D$-equivariant isomorphism $\Lambda(G^1) \xrightarrow{\sim} \Lambda(D^1)$
\end{proof}
\begin{defn}[c.f. Definition 3.9 in \cite{Ho1}]
Let $K$ be a number field. We define the \textbf{Kummer container} associated to $K$ to be the fibre product in the following diagram:
$$
\begin{tikzcd}
& \mathbb{I}_K^{\text{fin}} \arrow[d,hookrightarrow]\\
(K^{\times})^{\wedge} \arrow[r,hookrightarrow] & \prod_{\mathfrak{p} \in \mathscr{P}_K^{\text{fin}}} (K_{\mathfrak{p}}^{\times})^{\wedge}
\end{tikzcd}
$$
where the horizontal injection is defined in Theorem 9.1.11 (i) in \cite{NSW}, and denote it by $\mathcal{H}^{\times}(K)$. Thus,
$$
\mathcal{H}^{\times}(K) = (K^{\times})^{\wedge} \times_{\prod_{\mathfrak{p} \in \mathscr{P}_K^{\text{fin}}} (K_{\mathfrak{p}}^{\times})^{\wedge}} \mathbb{I}_K^{\text{fin}}.
$$
Moreover, we write
$$
\mathcal{H}_{\times}(K) := \mathcal{H}^{\times}(K) \cup \{0\}.
$$
\end{defn}
In our case, the Kummer sequence
$$
1 \to \mu_n(K_m) \to K_m^{\times} \xrightarrow{(-)^n} K_m^\times
$$
is not right exact for any positive integer $m$ and integers $n \geq 2$. However, we can still construct a group-theoretic version of the Kummer container in our setting using the following Lemma.
\begin{lem}
Let $K$ be a field of characteristic $0$. Let $m \geq 2$ be an integer. Then for any $n \in \mathbb{Z}_{\geq 1}$, we have a natural isomorphism:
$$
H^1(G_K^m,\mu_n) \xrightarrow{\sim} H^1(G_K,\mu_n).
$$
In particular, $H^1(G_K^m,\mu_n) \xrightarrow{\sim} K^{\times}/(K^{\times})^n$ for $m \geq 2$.
\end{lem}
\begin{proof}
Consider the inflation-restriction sequence (e.g. see Corollary 7.2.5 in \cite{RZ}):
$$
1 \to H^1(G_K^{\text{ab}},\mu_n^{G_K^{[m,1]}}) \to H^1(G_K^m,\mu_n) \to H^1(G_K^{[m,1]},\mu_n)^{G_K^{\text{ab}}} \to H^2(G_K^{\text{ab}}, \mu_n^{G_K^{[m,1]}}).
$$
Since $K^{\text{ab}} = K_1$ contains all $n$-th roots of unity for $n\geq 1$, the action of $G_K^{[m,1]}$ on $\mu_n$ is trivial. Hence we have the exact sequence
$$
1 \to H^1(G_K^{\text{ab}},\mu_n) \to H^1(G_K^m,\mu_n) \to \text{Hom}(G_K^{[m,1]},\mu_n)^{G_K^{\text{ab}}} \to H^2(G_K^{\text{ab}},\mu_n)
$$
which fits in the following commutative diagram with exact rows:
$$
\begin{tikzcd}
	1 & {H^1(G_K^{\text{ab}},\mu_n)} & {H^1(G_K^m,\mu_n)} & {\text{Hom}(G_K^{[m,1]},\mu_n)^{G_K^{\text{ab}}}} & {H^2(G_K^{\text{ab}},\mu_n)} \\
	1 & {H^1(G_K^{\text{ab}},\mu_n)} & {H^1(G_K,\mu_n)} & {\text{Hom}(G_K^{[1]},\mu_n)^{G_K^{\text{ab}}}} & {H^2(G_K^{\text{ab}},\mu_n)}
	\arrow[from=1-1, to=1-2]
	\arrow[from=1-2, to=1-3]
	\arrow[no head, from=1-2, to=2-2]
	\arrow[shift right, no head, from=1-2, to=2-2]
	\arrow[from=1-3, to=1-4]
	\arrow["{\text{inf}}", from=1-3, to=2-3]
	\arrow[from=1-4, to=1-5]
	\arrow["{\text{inf}}", from=1-4, to=2-4]
	\arrow[shift right, no head, from=1-5, to=2-5]
	\arrow[no head, from=1-5, to=2-5]
	\arrow[from=2-1, to=2-2]
	\arrow[from=2-2, to=2-3]
	\arrow[from=2-3, to=2-4]
	\arrow[from=2-4, to=2-5]
\end{tikzcd}
$$
where the right vertical inflation map
$$
\text{Hom}(G_K^{[m,1]},\mu_n)^{G_K^{\text{ab}}} \to \text{Hom}(G_K^{[1]},\mu_n)^{G_K^{\text{ab}}}
$$
is an isomorphism since $G_K^{[m,1]}$ and $G_K^{[1]}$ have isomorphic abelianisations. Thus, we can conclude that the left vertical inflation map is an isomorphism. By Kummer theory, we can conclude that
$$
H^1(G_K^m,\mu_n) \xrightarrow{\sim} K^{\times}/(K^{\times})^n. 
$$
\end{proof}
Now let $G$ be a profinite group of $\text{GSC}^6$-type, and let $D \in \widetilde{\text{Dec}}(G^4)$. Moreover, let $v \in \text{Dec}(G^1)$ be the image of $D$ via the surjection $\widetilde{\text{Dec}}(G^4) \twoheadrightarrow \text{Dec}(G^1)$. Let $D' \in \widetilde{\text{Dec}}(G^4)$ be an element whose image in $\text{Dec}(G^1)$ is $v$. Then the action of $G^4$ on $\widetilde{\text{Dec}}(G^4)$ induces a natural isomorphism
$$
H^1(D,\Lambda(D^1)) \xrightarrow{\sim} H^1(D',\Lambda((D')^1)).
$$
In this case, we may write
$$
H^1(v,\Lambda(v)) := H^1(D,\Lambda(D^1)).
$$
By Theorem 2.6, $\Lambda(D^1)$ is independent from the choice of $D$ whose image in $\text{Dec}(G^1)$ is $v$, hence we may write $\Lambda(v) := \Lambda(D^1)$. Thus, $H^1(v,\Lambda(D^1))$ is well-defined up to canonical isomorphism. \par 
By Theorem 3.7 and Theorem 9.1.11 (i) in \cite{NSW}, we have an injective homomorphism
$$
\dagger:H^1(G^4,\Lambda(G^1)) \hookrightarrow \prod_{v \in \text{Dec}(G^1)} H^1(v,\Lambda(v)).
$$
\begin{defn}
Let $G$ be a profinite group of $\text{GSC}^6$-type. We write
$$
\mathcal{H}^{\times}(G^1) \subset H^1(G^4,\Lambda(G^1))
$$
for the inverse image via the injective map $\dagger$ of $\mathbb{I}^{\text{fin}}(G^1)$, via the following composite map
$$
\mathbb{I}^{\text{fin}}(G^1) \hookrightarrow \prod_{v \in \text{Dec}(G^1)} k^{\times}(v) \hookrightarrow \prod_{v \in \text{Dec}(G^1)} H^1(v,\Lambda(v)),
$$
where the injectivity of 
$$
\prod_{v \in \text{Dec}(G^1)} k^{\times}(v) \hookrightarrow \prod_{v \in \text{Dec}(G^1)} H^1(v,\Lambda(v))
$$
follows from the injectivity of the Kummer map (c.f. Lemma 1.3 (x) in \cite{Ho1})
$$
\text{Kmm}_v:k^{\times}(v) \hookrightarrow H^1(v,\Lambda(v)).
$$
We call $\mathcal{H}^{\times}(G^1)$ the \textbf{Kummer container} associated to $G^1$. Moreover, we write 
$$
\mathcal{H}_{\times}(G^1) := \mathcal{H}^{\times}(G^1) \cup \{0\}.
$$
\end{defn}
\begin{rem}
In order to construct $\mathcal{H}^{\times}(G^1)$, one needs to reconstruct $\Lambda(G^1)$ and $\Lambda(D^1)$. Hence it follows from Theorem 2.6 (together with the various Definitions involved) that $\mathcal{H}^{\times}(G^1)$ can be group-theoretically reconstructed from a profinite group $G$ of $\text{GSC}^6$-type.
\end{rem}
\begin{prop}
Let $G$ be a profinite group of $\text{GSC}^7$-type, and let $H \subset G^1$ be an open subgroup. We write $\widetilde{H}$ for the inverse image of $H$ in $G$, hence $\widetilde{H}^6$ is a profinite group of $\text{GSC}^6$-type. Then we have a commutative diagram:
$$
\begin{tikzcd}
	{\mathcal{H}^{\times}(G^1)} & {\mathcal{H}^{\times}(K)} \\
	{\mathcal{H}^{\times}(\widetilde{H}^1)} & {\mathcal{H}^{\times}(L)} 
	\arrow["\sim", from=1-1, to=1-2]
	\arrow[hook, from=1-1, to=2-1]
	\arrow[hook, from=1-2, to=2-2]
	\arrow["\sim", from=2-1, to=2-2]
\end{tikzcd}
$$
where $L/K$ is the finite abelian extension determined by $H \subset G^1$, the horizontal arrows are isomorphisms, and the vertical arrows are injective.
\end{prop}
\begin{proof}
The right vertical arrow is induced by the commutative diagram (induced by the field embedding $K \hookrightarrow L$)
$$
\begin{tikzcd}
	{\mathbb I^{\text{fin}}_K} & {\mathbb{I}_L^{\text{fin}}} \\
	{\prod_{\mathfrak{p} \in \mathscr{P}_K^{\text{fin}}}(K_{\mathfrak{p}}^{\times})^{\wedge}} & {\prod_{\mathfrak{q} \in \mathscr{P}_L^{\text{fin}}}(L_{\mathfrak{q}}^{\times})^{\wedge}}
	\arrow[hook, from=1-1, to=1-2]
	\arrow[hook, from=1-1, to=2-1]
	\arrow[hook, from=1-2, to=2-2]
	\arrow[hook, from=2-1, to=2-2]
\end{tikzcd}
$$
together with 
$$
(K^{\times})^{\wedge} \hookrightarrow (L^{\times})^{\wedge}.
$$
The top horizontal arrow is induced by the following natural isomorphisms (c.f. Definition 3.2 and Lemma 3.9)
$$
\mathbb{I}^{\text{fin}}(G^1) \xrightarrow{\sim} \mathbb{I}^{\text{fin}}_K~;~H^1(G^4,\Lambda(G^1)) \xrightarrow{\sim} (K^{\times})^{\wedge}~\text{and}~\prod_{v \in \text{Dec}(G^1)}H^1(v,\Lambda(v)) \xrightarrow{\sim} \prod_{\mathfrak{p} \in \mathscr{P}_K^{\text{fin}}} (K_{\mathfrak{p}}^{\times})^{\wedge}.
$$
The bottom horizontal arrow is constructed in the same way as the top horizontal arrow. \par 
Now we construct the left vertical arrow. Notice that the inclusion $H \hookrightarrow G^1$ induces an inclusion $\widetilde{H} \hookrightarrow G$ which gives rise to the transfer map (e.g. see Proposition 1.5.9 in \cite{NSW})
$$
\text{Ver}: G^1 \to \widetilde{H}^1.
$$
In particular, the following diagram commutes
$$
\begin{tikzcd}
\mathbb{I}^{\text{fin}}(\widetilde{H}^1) \arrow[r] & \widetilde{H}^1 \\
\mathbb{I}^{\text{fin}}(G^1) \arrow[r] \arrow[u] & G^1 \arrow[u,"\text{Ver}"].
\end{tikzcd}
$$
It follows from Definition 3.5 and Definition 3.6 that we obtain an $H$-equivariant isomorphism
$$
\mu(H) \xrightarrow{\sim} \mu(G^1)
$$
w.r.t the restriction of $\chi: G^1 \to \widehat{\mathbb{Z}}^{\times}$ to $H$, hence we also have an $H$-equivariant isomorphism 
$$
\Lambda(H) \xrightarrow{\sim} \Lambda(G^1).
$$
Consider the restriction map
$$
\text{res}: H^1(G,\Lambda(G^1)) \to H^1(\widetilde{H},\Lambda(H)).
$$
The inflation map
$$
\text{inf}: H^1(\widetilde{H}^r,\Lambda(H)) \to H^1(\widetilde{H},\Lambda(H))
$$
is an isomorphism for $2 \leq r \leq 6$ (c.f. Lemma 3.9 and its proof). Without the loss of generality, we take $r = 2$. Then we have a natural map
$$
\theta: H^1(G^2,\Lambda(G^1)) \to H^1(\widetilde{H}^2,\Lambda(H))
$$
obtained by forming the composite
$$
H^1(G^2,\Lambda(G^1)) \xrightarrow{\sim} H^1(G,\Lambda(G^1)) \xrightarrow{\text{res}} H^1(\widetilde{H},\Lambda(H)) \xrightarrow{\sim} H^1(\widetilde{H}^2,\Lambda(H)).
$$
Putting $\theta$, the natural map $\mathbb{I}^{\text{fin}}(G^1) \to \mathbb{I}^{\text{fin}}(H)$, and the natural map 
$$
\prod_{v \in \text{Dec}(G^1)} H^1(v,\Lambda(v)) \to \prod_{w \in \text{Dec}(\widetilde{H}^1)} H^1(w,\Lambda(w))
$$
[induced by the restriction maps $H^1(D,\Lambda(D^1)) \to H^1(E,\Lambda(E^1))$ for suitable $D \in \widetilde{\text{Dec}}(G^2)$ and $E \in \widetilde{\text{Dec}}(\widetilde{H}^2)$] together, we obtain a natural map $\mathcal{H}^{\times}(G^1) \hookrightarrow \mathcal{H}^{\times}(\widetilde{H}^1)$. \par
One checks immediately that the right vertical map is injective, and the horizontal arrows are isomorphisms. The injectivity of the left vertical arrow follows immediately from the injectivity of the right vertical arrow. The commutativity of the diagram follows from the various definitions involved.
\end{proof}
\section{Reconstruction of $\mathbb{Q}_{m+3}$}
In this section we will adapt Hoshi's method in \cite{Ho2} in order to reconstruct the field extension $\mathbb{Q}_{m+3}/\mathbb{Q}$ inside $K_{m+3}$ for some integer $m \geq 0$. The point of this construction is to reconstruct the group $G_{\mathbb{Q}}^{m+3}$ starting from $G (\xrightarrow{\sim} G_K^{m+9})$. This construction is essential in order to recover enough local information to establish the final reconstruction. Moreover, here we need $m+3$ rather than $m+i$ for some $i\leq 2$ because we need to apply the statement of Theorem 1.2, where the three extra steps are essential.\par 
First, we shall characterise $\mathbb{Q}_{m+3}$ inside $K_{m+3}$. 
\begin{defn}
Let $K$ be a number field, and let $L/K$ be a finite extension contained in $K_{m+3}$ for some integer $m \geq 0$. For each positive integer $n$, we construct the following two sets
$$
\mathcal{G}(L,n) \subset \mathcal{F}(L,n)
$$
contained in $L_{\times}$ as follows: 
\begin{itemize}
    \item For $n=1$, we define $\mathcal{G}(L,n) := \mathcal{F}(L,n) := \{1\}$.
    \item For $n \geq 2$, we define
    $$
    \mathcal{G}(L,n) := \{a \in \mathcal{H}_{\times}(L): \exists N \geq 1~s.t. ~a^N \in \mathcal{F}(L,n-1)\}.
    $$
    Notice that $\mathcal{G}(L,n)$ is contained in $L_{\times}$ by Lemma 1.2 in \cite{Ho2}.
    \item For $n \geq 2$, we define $\mathcal{F}(L,n)$ to be the subfield of $L$ generated by $\mathcal{G}(L,n)$.
\end{itemize}
Moreover, we shall write
$$
\mathcal{F}(\infty,n) := \bigcup_L \mathcal{F}(L,n)
$$
where $L$ ranges over all finite extensions of $K$ contained in $K_{m+3}$, for some fixed positive integer $n$.
\end{defn}
\begin{lem}
Let $K$ be a number field, and let $m \geq 0$ be an integer. Then the equality
$$
\mathcal{F}(\infty,m+4) = \mathbb{Q}_{m+3}
$$
holds in $K_{m+3}$ (where the right hand side is clearly contained in $K_{m+3}$, and the left hand side is contained in $K_{m+3}$ by Lemma 1.2 in \cite{Ho2}).
\end{lem}
\begin{proof}
Consider $\mathcal{F}(\infty,2)$. It follows from Lemma 3.10 (iii) in \cite{Ho1}, together with the Kronecker-Weber Theorem, that
$$
\mathcal{F}(\infty,2) = \mathbb{Q}_1 (= \mathbb{Q}(\mu_{\infty})).
$$
By Definition 4.1, and Kummer theory, $\mathcal{F}(\infty,3)$ is an abelian extension of $\mathcal{F}(\infty,2) = \mathbb{Q}_1$. We can conclude that $\mathcal{F}(\infty,3)$ is contained in $\mathbb{Q}_2$ because $\mathbb{Q}_{2}$ is the maximal abelian extension of $\mathbb{Q}_1$. \par 
Conversely, let $F/\mathbb{Q}_1$ be a finite abelian extension. By Kummer theory, 
$$
F = \mathbb{Q}_1(a_1^{1/r_1},\dots,a_s^{1/r_s})
$$
where $a_1,\dots,a_s \in \mathbb{Q}_1$, $r_1,\dots,r_s \in \mathbb{Z}_{\geq2}$, and $s \in \mathbb{Z}_{\geq1}$. Then it follows immediately from Definition 4.1 that $F$ is contained in $\mathcal{F}(\infty,3)$. Thus, we can conclude that
$$
\mathcal{F}(\infty,3) = \mathbb{Q}_{2}.
$$
Now we consider $m = 0$. By Kummer theory and Definition 4.1, $\mathcal{F}(\infty,4)$ is an abelian extension of $\mathcal{F}(\infty,3)$, hence contained in $\mathbb{Q}_3$. Conversely, by Kummer theory, every finite abelian extension of $\mathbb{Q}_2$ is contained in $\mathcal{F}(\infty,4)$, hence $\mathbb{Q}_3$ is contained in $\mathcal{F}(\infty,4)$. Thus,
$$
\mathcal{F}(\infty,m+4) = \mathbb{Q}_{m+3}
$$
holds true for $m = 0$. Assume that
$$
\mathcal{F}(\infty,m+4) = \mathbb{Q}_{m+3}
$$
holds true for $m := j\geq 1$. Now we consider $m := j+1$. Again by Kummer theory and Definition 4.1, $\mathcal{F}(\infty,(j+1)+4)$ is an abelian extension of $\mathcal{F}(\infty,j+4)$, hence $\mathcal{F}(\infty,(j+1)+4)$ is contained in $\mathbb{Q}_{(j+1)+3}$. Conversely, by Kummer theory, every finite abelian extension of $\mathbb{Q}_{j+3}$ is contained in $\mathcal{F}(\infty,(j+1)+4)$, hence $\mathbb{Q}_{(j+1)+3}$ is contained in $\mathcal{F}(\infty,(j+1)+4)$. Thus, we may deduce that
$$
\mathcal{F}(\infty,j+4) = \mathbb{Q}_{j+3} \implies \mathcal{F}(\infty,(j+1)+4) = \mathbb{Q}_{(j+1)+3}.
$$
It follows then by induction that
$$
\mathcal{F}(\infty,m+4) = \mathbb{Q}_{m+3}
$$
for all $m \in \mathbb{Z}_{\geq 0}$.
\end{proof}
\begin{defn}
Let $G$ be a profinite group of $\text{GSC}^{m+2}$-type for some integer $m \geq 1$. We write $\widetilde{\text{Dec}}^{d=1}(G^{m})$ for the subset of $\widetilde{\text{Dec}}(G^{m})$ consisting of those elements $D$ such that $d_D = 1$. Moreover, we write $\text{Dec}^{d=1}(G^1)$ for the subset of $\text{Dec}(G^1)$ consisting of elements $v$ with $d_v = 1$.
\end{defn}
\begin{prop}
Let $G$ be a profinite group of $\text{GSC}^6$-type. Then for each $D \in \widetilde{\text{Dec}}^{d=1}(G^{4})$, one can group-theoretically reconstruct a (topological) field $k(D^1)$ from $D$ such that
$$
k(D^1) \xrightarrow{\sim} \mathbb{Q}_{p_D}.
$$
In this case, for $v \in \text{Dec}^{d=1}(G^1)$, we write 
$$
k(v) := k(D^{1})
$$
for any $D \in \widetilde{\text{Dec}}^{d=1}(G^4)$ whose image in $\text{Dec}^{d=1}(G^1)$ is $v$ [note that $k(v)$ does not depend on the choice of $D$ as the construction of the local cyclotomes does not depend on the choice of $D$, c.f. Theorem 2.6 (ix) and the proof below].
\end{prop}
\begin{proof}
This follows from Theorem 2.6 (ix) and Remark 2.1.1 in \cite{Ho2}. More precisely, we have the natural ring isomorphism
$$
\text{End}(\Lambda(D^1)^{(p_D)}) \xrightarrow{\sim} \mathbb{Z}_{p_D}, 
$$
which determines a ring structure of $\mathbb{Z}_{p_D}$. Similarly, we have a natural ring isomorphism
$$
\text{End}((\Lambda(D^1)^{(p_D)})^{\text{pf}}) \xrightarrow{\sim} \mathbb{Q}_{p_D}
$$
where $(-)^{\text{pf}}$ denotes the perfection of monoids (c.f. section 0 in \cite{Ho1}).
\end{proof}
Next, we reconstruct group-theoretically the minimal subfield $\mathbb{Q}$ of $K$.
\begin{prop}
Let $G$ be a profinite group of $\text{GSC}^6$-type. Then one can group-theoretically reconstruct a field $Q(G)$ from $G$ such that
$$
Q(G) \xrightarrow{\sim} \mathbb{Q}.
$$
\end{prop}
\begin{proof}
This is essentially the same argument as in Definition 2.1 and Definition 2.2 in \cite{Ho2}. More precisely, by Proposition 4.4, for each $D \in \widetilde{\text{Dec}}^{d=1}(G^4)$, one can reconstruct group-theoretically a unique minimal subfield $Q(v)$ of $k(D^1)$ isomorphic to $\mathbb{Q}$, where $v \in \text{Dec}^{d=1}(G^1)$ is the image of $D$ (and this definition is independent from the choice of $D$ whose image is $v$). Then we define
$$
Q(G) \subset \mathcal{H}_{\times}(G^1) \subset \prod_{v \in \text{Dec}^{d=1}(G^1)} k_{\times}(v)
$$
(where the second inclusion follows from Lemma 3.10 (v) in \cite{Ho1}) to be the subset of $\mathcal{H}_{\times}(G^1)$ consisting of elements $a \in \mathcal{H}_{\times}(G^1)$ such that the image of $a$ in $k_{\times}(v)$ and in $k_{\times}(w)$ are naturally identified via the unique isomorphism
$$
Q(v) \xrightarrow{\sim} Q(w)
$$
for all $v,w \in \text{Dec}^{d=1}(G^1)$. Then we may view
$$
Q(G) \subset \prod_{v \in \text{Dec}^{d=1}(G)} Q(v)
$$
as a subring satisfying the property above. One checks immediately that the isomorphism of monoids $\mathcal{H}_{\times}(G^1) \xrightarrow{\sim} \mathcal{H}_{\times}(K)$ (c.f. Proposition 3.12) restricts to an isomorphism
$$Q(G) \xrightarrow{\sim}\mathbb{Q}.$$
\end{proof}
\begin{defn}
Let $G$ be a profinite group of $\text{GSC}^{m+9}$-type for $m \geq 0$. Let $H \subset G^{m+3}$ be an open subgroup and write $\widetilde{H}$ for the inverse image of $H$ in $G$. Thus, $\widetilde{H}^6$ is of $\text{GSC}^6$-type. For each positive integer $n$, we define the following two sets
$$
\mathcal{G}(H,n) \subset \mathcal{F}(H,n)
$$
contained in $\mathcal{H}_{\times}(\widetilde{H}^1)$ as follows:
\begin{itemize}
    \item For $n = 1$, we define $\mathcal{G}(H,n) := \mathcal{F}(H,n) := \{1\}$.
    \item For $n \geq 2$, we define
    $$
    \mathcal{G}(H,n) := \{a \in \mathcal{H}_{\times}(\widetilde{H}^1):\exists N \geq 1~s.t.~a^N \in \mathcal{F}(H,n-1)\}.
    $$
    \item For $n \geq 2$, we define $\mathcal{F}(H,n)$ to be the subring of
    $$
    \prod_{v \in \text{Dec}^{d=1}(\widetilde{H}^1)} k(v)
    $$
    (where the injectivity of $\mathcal{H}_{\times}(\widetilde{H}^1) \subset \prod_{v \in \text{Dec}^{d=1}(\widetilde{H}^1)} k(v)$ follows from Lemma 3.10 (v) in \cite{Ho1}) generated by $\mathcal{G}(H,n)$.
\end{itemize}
Moreover, we define
$$
\mathcal{F}_{G^{m+3}}(\infty,n) := \varinjlim_H ~\mathcal{F}(H,n)
$$
where $H$ ranges over all open subgroups of $G^{m+3}$, and the transition maps are defined to be natural inclusions (c.f. Proposition 3.12)
$$
\mathcal{H}_{\times}(\widetilde{H}^1) \hookrightarrow \mathcal{H}_{\times}(\widetilde{H'}^1)
$$
for each open subgroups $H' \subset H \subset G^{m+3}$.
\end{defn}
\begin{prop}
Let $G$ be a profinite group of $\text{GSC}^{m+9}$-type for some integer $m \geq 0$. Then the followings hold: \par 
(i) There exists a field isomorphism
$$
\mathcal{F}_{G^{m+3}}(\infty,m+4) \xrightarrow{\sim} \mathcal{F}(\infty,m+4)
$$
which is equivariant w.r.t the isomorphism $\alpha_{m+3}: G^{m+3} \xrightarrow{\sim} G_K^{m+3}$. \par
(ii) We write $\mathcal{Q} := \text{Aut}_{\text{field}}(\mathcal{F}_{G^{m+3}}(\infty,m+3))$ for the automorphism group of $\mathcal{F}_{G^{m+3}}(\infty,m+3)$. Then the natural action of $G^{m+3}$ on $\mathcal{F}_{G^{m+3}}(\infty,m+4)$ by automorphism determines a homomorphism
$$
G^{m+3} \to \mathcal{Q}.
$$
\end{prop}
\begin{proof}
We first verify assertion (i). It follows from Proposition 3.12 that we have the following commutative diagram
$$
\begin{tikzcd}
\mathcal{H}_{\times}(G^1) \arrow[d,hookrightarrow] \arrow[r,"\sim"] & \mathcal{H}_{\times}(K) \arrow[d,hookrightarrow] \\
\mathcal{H}_{\times}(\widetilde{H}^1) \arrow[r,"\sim"] & \mathcal{H}_{\times}(L)
\end{tikzcd}
$$
where $H \subset G^{m+3}$ is the open subgroup corresponding to the finite extension $L/K$ contained in $K_{m+3}$ via $\alpha_{m+3}$. \par 
For each open normal subgroup $H \subset G^{m+3}$ corresponding to some finite Galois extension $L/K$ contained in $K_{m+3}$ via $\alpha_{m+3}$, we have the following commutative diagram (where the injectivity of the lower vertical arrows follows from Lemma 3.10 (v) in \cite{Ho1})
$$
\begin{tikzcd}
	{\mathcal{F}(H,m+4)} && {\mathcal{F}(L,m+4)} \\
	{\mathcal{H}_{\times}(\widetilde{H}^1)} && {\mathcal{H_{\times}}(L)} \\
	{\prod_{v \in \text{Dec}^{d=1}(\widetilde{H}^1)}k(v)} && {\prod_{\mathfrak p \in \mathscr{P}_L^{\text{fin,d=1}}}L_{\mathfrak p}}
	\arrow["\sim", from=1-1, to=1-3]
	\arrow[hook, from=1-1, to=2-1]
	\arrow[hook, from=1-3, to=2-3]
	\arrow["\sim", from=2-1, to=2-3]
	\arrow[hook, from=2-1, to=3-1]
	\arrow[hook, from=2-3, to=3-3]
	\arrow["\sim", from=3-1, to=3-3]
\end{tikzcd}
$$
where the top horizontal isomorphism is obtained by restricting the middle isomorphism (c.f. Definition 4.1 and Definition 4.6). Moreover, the top isomorphism is an isomorphism of rings, the preservation of the ring structures of the top isomorphism follows from the fact that the bottom isomorphism is an isomorphism of rings. \par 
Further, the isomorphism
$$
\mathcal{F}(H,n) \xrightarrow{\sim} \mathcal{F}(L,n)
$$
is equivariant w.r.t the isomorphism $G^{m+3}/H \xrightarrow{\sim} \text{Gal}(L/K)$ induced by $\alpha_{m+3}$. By taking inductive limit over all open normal subgroups of $H$ of $G^{m+3}$, we obtain an isomorphism of fields
$$
\mathcal{F}_{G^{m+3}}(\infty,m+4) \xrightarrow{\sim} \mathcal{F}(\infty,m+4)
$$
which is equivariant w.r.t $\alpha_{m+3}$. This proves assertion (i), assertion (ii) is an immediate consequence of assertion (i).
\end{proof}
\begin{prop}
Let $G$ be a profinite group of $\text{GSC}^{m+9}$-type for $m \geq 0$. Let $H \subset G^{m+3}$ be an open subgroup and write $\widetilde{H}$ for the inverse image of $H$ in $G$. Then there exists a natural injective map
$$
\mathcal{F}(H,n)^{\times} \to k^{\times}( D \cap H)
$$
for any $D \in \widetilde{\text{Dec}}(G^{m+3})$ and each positive integer $n$.
\end{prop}
\begin{proof}
Notice that to prove the assertion, it suffice to verify that
$$
\mathcal{G}^{\times}(H,n) := \mathcal{G}(H,n) \setminus \{0\}
$$
is contained in $k^{\times}(D \cap H)$ (indeed, it follows from Proposition 4.5 and Definition 4.6 that $\mathcal{F}(H,n)$ is a field extension of $Q(G)$). We shall verify this by induction on $n$. \par
For $n =1$, $\mathcal{G}^{\times}(H,n) = \{1\}$ is trivially contained in $k^{\times}(D \cap H)$. Let $n = 2$, by Definition 4.6 and Lemma 3.10 (iii) in \cite{Ho1}, $\mathcal{G}^{\times}(H,2) = \mathcal{H}^{\times}(\widetilde{H}^1)_{\text{tor}} = \mu(H)^H$. Then it follows from Theorem 3.7 and its proof, that $\mu(H)^H \subset k^{\times}(D \cap H)_{\text{tor}} \subset k^{\times}(D \cap H)$. \par 
Assume that for $n = i \geq 2$, the containment $\mathcal{G}^{\times}(H,i) \subset k^{\times}(D \cap H)$ holds true. Hence there exists a natural injective map
$$
\eta_i:\mathcal{F}(H,i)^{\times} \hookrightarrow k^{\times}(D \cap H).
$$
By Definition 4.6, let $ a \in \mathcal{G}^{\times}(H,i+1)$, so $\exists N \in \mathbb N$ such that $a^N \in \mathcal{F}(H,i)$. In particular, $a^N \in k^{\times}(D \cap H)$. On the other hand, it follows from Lemma 1.2 in \cite{Ho2} together with the various definitions involved, that $a$ must be contained in $k^{\times}(D \cap H)$. Hence by induction, $\mathcal{G}^{\times}(H,n) \subset k^{\times}(D \cap H)$ for arbitrary positive integer $n$.
\end{proof}
\section{Reconstruction of Number Fields}
In this section we shall prove Theorem 1. In order to achieve this, we have to assume $m\geq 3$, because Proposition 5.4 only works for $m\geq3$, and Proposition 5.4 is essential in order to prove Lemma 5.5. Nevertheless, there are still some results in this section which hold true for $m=0$ (this will be specified in their statements). Throughout this section we let $G$ be a profinite group of $\text{GSC}^{m+9}$-type. \par 
Let $D \in \widetilde{\text{Dec}}^{d=1}(G^{m+3})$ (as in section 4, we need $m+3$ in order to apply the statement of Theorem 1.2 for Proposition 5.6). First, we shall reconstruct group-theoretically the field $k(D)_{m+3}$ starting from $G$. \par 
Notice that $d_D = 1$, hence by Proposition 1.1 (i) in \cite{ST1} we have the identification
$$
D \xrightarrow{\sim} \text{Gal}((\mathbb{Q}_{p_D})_{m+3}/\mathbb{Q}_{p_D}).
$$
\begin{lem}
Let $m \geq 0$ and let $D \in \widetilde{\text{Dec}}^{d=1}(G^{m+3})$. The composite
$$
D \hookrightarrow G^{m+3} \to \mathcal{Q} ~(\text{c.f. Proposition 4.7 (ii)})
$$
is injective, where the second arrow is determined by the natural action of $G^{m+3}$ on $\mathcal{F}_{G^{m+3}}(m+4)$ by automorphisms.
\end{lem}
\begin{proof}
By Proposition 4.7 (i), we have the following commutative diagram
$$
\begin{tikzcd}
D \arrow[r,hookrightarrow] \arrow[d,"\alpha_D"] & G^{m+3} \arrow[r] \arrow[d,"\alpha"] & \mathcal{Q} \arrow[d,"\wr"] \\
D_{\mathfrak{p}} \arrow[r,hookrightarrow] & G_{K}^{m+3} \arrow[r] & G_{\mathbb{Q}}^{m+3}
\end{tikzcd}
$$
where $\alpha_D$ is also an isomorphism by Theorem 1.25 in \cite{ST1} and $\mathfrak{p} \in \mathscr{P}_{K_{m+3}}^{\text{fin,d=1}}$. Moreover, the arrow $G_K^{m+3} \to G_{\mathbb{Q}}^{m+3}$ is defined to be the composite 
$$
G_K^{m+3} \hookrightarrow \text{Aut}(K_{m+3}) \twoheadrightarrow \text{Aut}(\mathbb{Q}_{m+3}) \xrightarrow{\sim} G_{\mathbb{Q}}^{m+3}
$$
where the first arrow is the natural inclusion, the second surjection is obtained by restricting automorphisms of $K_{m+3}$ to $\mathbb{Q}_{m+3}$ and the third isomorphism is the identity. \par
Thus, verifying that the upper composite is injective is equivalent to verifying the lower composite is injective. \par 
By definition the map $G_K^{m+3} \to G_{\mathbb Q}^{m+3}$ arises from a field embedding $\mathbb{Q} \hookrightarrow K$, hence preserves decomposition subgroups. In particular, since $d_\mathfrak{p} = 1$, $D_{\mathfrak{p}}$ is mapped isomorphically onto some decomposition subgroup of $G_{\mathbb{Q}}^{m+3}$ via $G_K^{m+3} \to G_{\mathbb{Q}}^{m+3}$. Therefore, the composite
$$
D_{\mathfrak{p}} \hookrightarrow G_K^{m+3} \to G_{\mathbb{Q}}^{m+3}
$$
is injective. 
\end{proof}
\begin{prop}
Let $m \geq 0$ and let $D \in \widetilde{\text{Dec}}^{d=1}(G^{m+3})$. There is a group-theoretic reconstruction of the field $k(D)_{m+3}$ starting from $G$. In particular, we have a $\alpha_{m+3}|_D$-equivariant isomorphism of fields
$$
k(D)_{m+3} \xrightarrow{\sim} (\mathbb{Q}_{p})_{m+3}.
$$
\end{prop}
\begin{proof}
We adapt the proof of Proposition 5.8 in \cite{Ho1} to our case. \par 
The group $D$ is isomorphic to a quotient of $G_{\mathbb{Q}_{p_D}}$. Since $G_{\mathbb{Q}_{p_D}}$ is topologically finitely generated by Theorem 7.5.14 in \cite{NSW}, $D$ is also topologically finitely generated. \par 
By Proposition 2.5.1 in \cite{RZ}, $D$ admits a fundamental system of open characteristic neighborhoods of $1$, i.e. a series of open characteristic subgroups
$$
D := D_0 \supset D_1 \dots \supset D_n \supset \dots
$$
such that $\bigcap_{i \geq 0} D_i = \{1\}$ where $i \in \mathbb{Z}_{\geq 0}$. Then we write 
$$
\mathcal{Q} = \mathcal{Q}_0 \supset \mathcal{Q}_1 \supset \mathcal{Q}_2 \dots \supset \mathcal{Q}_n \supset \dots
$$
where $\mathcal{Q}_n$ is an open subgroup of $\mathcal{Q}$ such that the equality $D_n = \mathcal{Q}_n \cap D$ holds. In particular, the family $\{\mathcal{Q}_n\}_{n \geq 1}$ is not unique.
By possibly replacing $\mathcal{Q}_n$ by some open normal subgroup $\mathcal{Q}_n'$ of $\mathcal{Q}$ such that the equality $D_n = \mathcal{Q}_n' \cap D$ holds, we may assume that the $\{\mathcal{Q}_n\}_{n \geq 1}$ are open normal subgroups of $\mathcal{Q}$. \par 
By Proposition 4.7, we have reconstructed the field $\mathcal{F}_{G^{m+3}}(\infty,m+4)$ and we can use its multiplicative group. In particular, we write
$$
\mathcal{M}(\mathcal{Q}_n) := (\mathcal{F}_{G^{m+3}}(\infty,m+4)^{\times})^{\mathcal{Q}_n}.
$$
Now by Theorem 2.6, we can write $k^{\times}(D)_{m+3} := \varinjlim_{H} k^{\times}(H)$ where $H$ ranges over all open subgroups of $D$. Then we write
$$
\mathcal{M}(D_n) := (k^{\times}(D)_{m+3})^{D_n}.
$$
We have the following commutative diagram
$$
\begin{tikzcd}
\mathcal{F}_{G^{m+3}}(\infty,m+4)^{\times} \arrow[r,hookrightarrow] & k^{\times}(D)_{m+3} \\
(\mathcal{F}_{G^{m+3}}(\infty,m+4)^{\times})^{D_n} \arrow[u,hookrightarrow] \arrow[r,hookrightarrow] & (k^{\times}(D)_{m+3})^{D_n} \arrow[u,hookrightarrow] \\
(\mathcal{F}_{G^{m+3}}(\infty,m+4)^{\times})^{\mathcal{Q}_n} \arrow[u,hookrightarrow] &
\end{tikzcd}
$$
where the top horizontal arrow is defined by taking inductive limit on $H$ of the natural inclusion map presented in Proposition 4.8 for $n = m+4$, the bottom horizontal arrow is the injective map induced by the top horizontal arrow, and the vertical arrows are natural inclusions. In particular, we have an inclusion $\mathcal{M}(\mathcal{Q}_n) \hookrightarrow \mathcal{M}(D_n)$.
\par 
On the other hand, we have the field structure of $\mathcal{F}_{G^{m+3}}(\infty,m+4)$. Thus, if we write $\mathcal{M}^*(\mathcal{Q}_n) := \mathcal{M}(\mathcal{Q}_n) \cup \{0\}$ (where $0$ is defined to be the unique element in $\mathcal{F}_{G^{m+3}}(\infty,m+4) \setminus \mathcal{F}_{G^{m+3}}(\infty,m+4)^{\times}$), we can equip $\mathcal{M}^*(\mathcal{Q}_n)$ with the field structure determined by $\mathcal{F}_{G^{m+3}}(\infty,m+4)$. In particular, $\mathcal{M}^*(\mathcal{Q}_n)$ is identified with the fixed subfield $\mathcal{F}_{G^{m+3}}(\infty,m+4)^{\mathcal{Q}_n}$ of $\mathcal{F}_{G^{m+3}}(\infty,m+4)$. \par 
Then we apply the method in Proposition 5.8 (1) in \cite{Ho1} in order to obtain the local fields $k(D_n)$. More precisely, for each $n$, we write $k_{\times}(D_n) := k^{\times}(D_n) \cup \{0\}$ where $k^{\times}(D_n) := (k^{\times}(D)_{m+3})^{D_n}$. Then we define the binary map
$$
+_{k(D_n)} : k_{\times}(D_n) \times k_{\times}(D_n) \to k_{\times}(D_n)
$$
such that the following hold:
\begin{itemize}
\item $+_{k(D_n)}(a,0) = +_{k(D_n)}(0,a) = a$ for all $a \in k_{\times}(D_n)$.
\item We write 
$$
\mathbf{Z}(D_n) := \{(a,b) \in k^{\times}(D_n) \times k^{\times}(D_n): ab^{-1} \neq 1~\text{but}~(ab^{-1})^2 = 1\}
$$
and 
$$
\mathbf{S}(D_n) := (k^{\times}(D_n) \times k^{\times}(D_n)) \setminus \mathbf{Z}(D_n).
$$
Then $+_{k(D_n)}(\mathbf{Z}(D_n)) = \{0\}$, where $+_{k(D_n)}(\mathbf{Z}(D_n))$ is the image of $\mathbf{Z}(D_n)$ under the binary map $+_{k(D_n)}$. 
\item For $(a,b) \in \mathbf{S}(D_n)$, it follows from the fact that number fields are dense inside their (metric) completions, that there is a sequence of elements
$$
(a_i,b_i) \in \mathcal{M}(\mathcal{Q}_n) \cap \mathbf{S}(D_n)
$$
such that $(a_i,b_i)$ converges to $(a,b)$. Then we write
$$
+_{k(D_n)}(a,b) := \lim_{i \to \infty} (a_i+b_i)
$$
where $a_i+b_i$ is taken inside $\mathcal{M}^*(\mathcal{Q}_n)$ for each $i$.
\end{itemize}
Thus we obtain a field structure on $k_{\times}(D_n)$ for each $n$. Consider the following diagram: 
$$
\begin{tikzcd}
k_{\times}(D_n) \times k_{\times}(D_n) \arrow[r,"+_{k(D_n)}"] \arrow[d,hookrightarrow] & k_{\times}(D_n) \arrow[d,hookrightarrow] \\
k_{\times}(D_{n+1}) \times k_{\times}(D_{n+1}) \arrow[r,"+_{k(D_{n+1})}"] & k_{\times}(D_{n+1}).
\end{tikzcd}
$$
We claim that this diagram is commutative. Indeed, both vertical maps are natural inclusions, it holds that for any $(a,b) \in k_{\times}(D_n) \times k_{\times}(D_n)$, $(a,b) \in k_{\times}(D_{n+1}) \times k_{\times}(D_{n+1})$. Hence one may conclude that the above diagram is commutative. We define $k_{\times}(D)_{m+3} := \varinjlim_n k_{\times}(D_n)$ where the transition maps are taken to be natural inclusions. By the commutativity of the above diagram, $+_{k(D_n)}$ is compatible with the transition maps for each $n \geq 1$. Write $+_{k(D)_{m+3}}$ for the inductive limit of $+_{k(D_n)}$, hence we obtain the binary map:\par 
$$
+_{k(D)_{m+3}}: k_{\times}(D)_{m+3} \times k_{\times}(D)_{m+3} \to k_{\times}(D)_{m+3}.
$$
It follows from the compatibility of $+_{k(D_n)}$ and the inductive limit, and the field structures of $k(D_n)$ for each $n \geq 1$, that $k_{\times}(D)_{m+3}$ equipped with $+_{k(D)_{m+3}}$ is indeed a field. Moreover, since $d_D = 1$, we have an $\alpha_{m+3}|_D$-equivariant isomorphism of fields $k(D)_{m+3} \xrightarrow{\sim} (\mathbb{Q}_{p_D})_{m+3}$. Finally, one verifies immediately that the construction of $k(D)_{m+3}$ is independent from the choice of the family $\{\mathcal{Q}_n\}_{n \geq 1}$.
\end{proof}
Now we want to characterise the image of $K_{m+3}$ in $k(D)_{m+3}$ in a group-theoretic way, and this is the central step in the proof of Theorem 1.
\begin{defn}
Let $j$ be a positive integer and let $r \geq 2$ be an integer. Let $G$ be a profinite group of $\text{GSC}^{j+r}$-type and $D \in \widetilde{\text{Dec}}^{d=1}(G^j)$. We define the following two subfields of $k(D)_j$:
$$
F[D] \subset F_j[D] \subset k(D)_j
$$
as follows: \par 
(1) $F[D]$ is a finite extension of the minimal subfield of $k(D)$ (i.e. a number field). \par 
(2) $F_j[D]$ is the maximal Galois extension of $F[D]$ contained in $k(D)_j$ such that every finite Galois extension $L/F[D]$ contained in $F_j[D]$ is solvable over $F[D]$ of length at most $j$. \par 
(3) Let $\sigma \in D$. The automorphism $\sigma: k(D)_j \xrightarrow{\sim} k(D)_j$ induces an automorphism $\sigma: F_j[D] \xrightarrow{\sim} F_j[D]$ and restricts to the identity automorphism on $F[D]$. \par 
(4) There exists a continuous isomorphism $\text{Gal}(F_j[D]/F[D]) \xrightarrow{\sim} G^j$ such that the following diagram
$$
\begin{tikzcd}
\text{Gal}(F_j[D]/F[D]) \arrow[r,"\sim"] & G^j \\
D \arrow[u,hookrightarrow] \arrow[ur,hookrightarrow] &
\end{tikzcd}
$$
where the inclusion $D \hookrightarrow \text{Gal}(F_j[D]/F[D])$ is induced by the field embedding $F_j[D] \hookrightarrow k(D)_j$, and $D \hookrightarrow G^j$ is the natural inclusion, is commutative. \par 
We say the pair of subfields $F[D] \subset F_j[D]$ is of \textbf{length $j$ standard type}.
\end{defn}
\begin{prop}
Let $m \geq 3$ be an integer. For each $D \in \widetilde{\text{Dec}}^{d=1}(G^{m})$, length $m$ standard type subfields of $k(D)_{m}$ exist. Moreover, if $F[D] \subset F_m[D]$ is a length $m$ standard type subfields, then $F[D]$ and $K$ (c.f. Definition 2.1) are isomorphic. 
\end{prop}
\begin{proof}
The first part of the proposition follows immediately from the definition of standard type subfields of length $m$. The second part of the proposition follows immediately from Theorem 1 in \cite{ST1}.
\end{proof}
We wish to show the uniqueness of length $m+3$ standard type subfields of $k(D)_{m+3}$, however, we can only show the uniqueness of length $m$ standard type subfields of $k(D)_m$ starting from length $m+3$ standard type subfields of $k(D)_{m+3}$.
\begin{lem}
Let $m \geq 3$ be an integer. Let $D \in \widetilde{\text{Dec}}^{d=1}(G^{m+3})$, and let $\mathfrak{D} \subset G^m$ be the image of $D$ in $G^m$ via the natural surjection $G^{m+3} \twoheadrightarrow G^m$, hence $\mathfrak{D} \in \widetilde{\text{Dec}}^{d=1}(G^m)$. Then a pair of length $m+3$ standard type subfields of $k(D)_{m+3}$ determines a pair of length $m$ standard type subfields of $k(\mathfrak{D})_m$ by taking the intersection $F_{m+3}[D] \cap k(\mathfrak{D})_m$, and every pair of length $m$ standard type subfields of $k(\mathfrak{D})_m$ arises from a pair of length $m+3$ standard type subfields of $k(D)_{m+3}$.
\end{lem}
\begin{proof}
To begin with, we show that a pair of length $m+3$ standard type subfields
$$
F[D] \subset F_{m+3}[D] \subset k(D)_{m+3}
$$
determines a pair of length $m$ standard type subfields of $k(\mathfrak{D})_m$. \par
We set $\mathscr{F} := F[D]$, and $\mathscr{F}_m$ to be the maximal Galois extension of $\mathscr{F}$ contained in $F_{m+3}[D]$ such that every intermediate (Galois) subextension $L/\mathscr{F}$ of $\mathscr{F}_m/\mathscr{F}$ is solvable of length at most $m$ over $\mathscr{F}$. Moreover, since $\mathfrak{D} \xrightarrow{\sim} D^m$, one has the commutative diagram
$$
\begin{tikzcd}
k(D)_{m+3} & \\
k(D)_m \arrow[u,hookrightarrow] \arrow[r,equal] & k(\mathfrak{D})_m \arrow[ul,hookrightarrow] \\
k(D) \arrow[u,hookrightarrow] \arrow[r,equal] & k(\mathfrak{D}) \arrow[u,hookrightarrow]
\end{tikzcd}
$$
and hence we have $\mathscr{F} \subset \mathscr{F}_m \subset k(\mathfrak{D})_m = k(D)_m$. Therefore, conditions (1) and (2) in Definition 5.3 are verified. Now take $\alpha \in D \subset \text{Gal}(F_{m+3}[D]/F[D])$, we write $\alpha_m$ for the image of $\alpha$ under the surjection $D \twoheadrightarrow \mathfrak{D}$. By construction, $k(\mathfrak{D})_m = k(D)_{m+3}^{\text{ker}(D \twoheadrightarrow \mathfrak{D})}$. Hence by Galois theory, $\alpha_m$ determines an automorphism $k(\mathfrak{D})_m \xrightarrow{\sim} k(\mathfrak{D})_m$. On the other hand, $\alpha$ determines an automorphism $F_{m+3}[D] \xrightarrow{\sim} F_{m+3}[D]$ that restricts to the identity automorphism of $F[D] = \mathscr{F}$, hence $\alpha_m$ also determines an automorphism of $\mathscr F_m$ which restricts to the identity automorphism on $\mathscr{F}$ and condition (3) is verified. Now consider the profinite group $\text{Gal}(\mathscr{F}_m/\mathscr{F})$. By construction, $\text{Gal}(\mathscr{F}_m/\mathscr{F})$ is identified with the maximal $m$-step solvable quotient of $\text{Gal}(F_{m+3}[D]/F[D])$. Then we consider the following diagram
$$
\begin{tikzcd}[row sep=scriptsize, column sep=scriptsize]
\text{Gal}(F_{m+3}[D]/F[D]) \arrow[rr,"\sim"] \arrow[dd,twoheadrightarrow] & & G^{m+3} \arrow[dd,twoheadrightarrow] \\
& D \arrow[ul,hookrightarrow] \arrow[ur,hookrightarrow] \arrow[dd,twoheadrightarrow, crossing over] & \\
\text{Gal}(\mathscr{F}_m/\mathscr{F}) \arrow[rr,"\sim"] & & G^m \\
& \mathfrak{D} \arrow[ur,hookrightarrow] \arrow[ul, hookrightarrow] \arrow[from=uu,twoheadrightarrow , crossing over] &
\end{tikzcd}
$$
where the inclusion $\mathfrak{D} \hookrightarrow \text{Gal}(\mathscr{F}_m/\mathscr{F})$ is determined by the field embedding $\mathscr{F} \subset k(\mathfrak{D})_m$, and the isomorphism $\text{Gal}(\mathscr{F}_m/\mathscr{F}) \xrightarrow{\sim} G^m$ is induced by the given isomorphism $\text{Gal}(F_{m+3}[D]/F[D]) \xrightarrow{\sim} G^{m+3}$. Then it follows from the commutativity of the top triangle and the three rectangles, that the bottom triangle is also commutative, thus condition (4) is verified. Hence we may conclude that a pair of length $m+3$ standard type subfields of $k(D)_{m+3}$ determines a pair of length $m$ standard type subfields of $k(\mathfrak{D})_m$. \par 
Conversely, given a pair of length $m$ standard type subfields
$$
F[\mathfrak{D}] \subset F_m[\mathfrak{D}] \subset k(\mathfrak{D})_m
$$
we want to find a pair of length $m+3$ standard type subfields
$$
F[D] \subset F_{m+3}[D] \subset k(D)_{m+3}
$$
which induces $F[D] \subset F_m[D]$ as above. We define $\widetilde{F}/F_m[\mathfrak{D}]$ to be the maximal subextension of $k(\mathfrak{D})_{m+3} (= k(D)_{m+3})$ such that every finite extension $L/F_m[\mathfrak{D}]$ contained in $\widetilde{F}$ is solvable over $F_m[\mathfrak{D}]$ of length at most $3$. In particular, $F[\mathfrak{D}] \subset \widetilde{F} \subset k(\mathfrak{D})_{m+3}$ is a pair of subfields satisfying conditions (1) and (2) in Definition 5.3. Let $\sigma \in \mathfrak{D}$ and fix a lifting $\tilde{\sigma} \in D$. Since $\widetilde{F} \hookrightarrow k(\mathfrak{D})_{m+3}$ can be obtained by extending the natural inclusion $F[\mathfrak{D}] \hookrightarrow k(\mathfrak{D})$, $\tilde{\sigma}$ determines an automorphism
$$
\tilde{\sigma}: \widetilde{F} \xrightarrow{\sim} \widetilde{F}
$$
that restricts to the identity automorphism of $F[\mathfrak{D}]$. Thus the pair $(F[\mathfrak{D}],\widetilde{F})$ satisfies condition (3) in Definition 5.3. Now we verify condition (4). Let $F_1/F_m[\mathfrak{D}]$ be an extension such that $F_1/F[\mathfrak{D}]$ is a maximal $m+3$-step solvable extension. Then it holds that every emebdding $F[\mathfrak{D}] \hookrightarrow k(\mathfrak{D})_{m+3}$ extends to $F_1 \hookrightarrow k(\mathfrak{D})_{m+3}$, more precisely, extends to the composite $F_1 \hookrightarrow (F_1)_{\tilde{v}} \xrightarrow{\sim} k(D)_{m+3} = k(\mathfrak{D})_{m+3}$ where $\tilde{v}$ is some non-archimedean prime of $F_1$ determined by $D$. Hence by the construction of $\widetilde{F}$, it holds that the image of $F_1$ in $k(\mathfrak{D})_{m+3}$ coincide with $\widetilde{F}$. Together with condition (3) (verified above), we obtain an embedding $D \hookrightarrow \text{Gal}(\widetilde{F}/F[\mathfrak{D}])$. On the other hand, by Proposition 5.4, we obtain an isomorphism $\text{Gal}(\widetilde{F}/F[\mathfrak{D}]) \xrightarrow{\sim} G^{m+3}$. Thus, the condition (4) in Definition 5.3 follows from the constructions above and $(F[\mathfrak{D}],\widetilde{F})$ is a pair of length $m+3$ standard type subfields.
\end{proof}
\begin{prop}
We keep the assumptions as in Lemma 5.5. Let $F[\mathfrak{D}] \subset F_m[\mathfrak{D}]$ be a pair of length $m$ standard type subfields of $k(\mathfrak{D})_m$ determined by $F_{m+3}[D]$. It holds that $F[\mathfrak{D}] \subset F_m[\mathfrak{D}]$ is unique.
\end{prop}
\begin{proof}
This proposition follows from Lemma 3.6 (ii) in \cite{Ho2} and Lemma 3.9 together with Theorem 1.2 and Lemma 5.5. More precisely, if
$$
(F[D]^{\circ},F_{m+3}[D]^{\circ}) ~\text{and}~(F[D]^{\bullet},F_{m+3}[D]^{\bullet})
$$
are two pairs of length $m+3$ standard type subfields, then Lemma 5.5 together with Theorem 1.2 implies that the length $m$ standard type subfields
$$
(F[\mathfrak{D}]^{\circ},F_m[\mathfrak{D}]^{\circ})~\text{and}~(F[\mathfrak{D}]^{\bullet},F_m[\mathfrak{D}^{\bullet}])
$$
must be isomorphic. Then we can use Lemma 3.9 in order to apply the arguments in the proof of Lemma 3.6 (ii) in \cite{Ho2} and obtain the uniqueness of the pair of length $m$ standard type subfields in Lemma 5.5.
\end{proof} 
We may now conclude that there is a group-theoretic reconstruction of length $m$ standard type subfields of $k(\mathfrak{D})_m$ for each $\mathfrak{D}$ starting from $G(\xrightarrow{\sim} G_K^{m+9})$ (and we need to take $m\geq 3$ in order to use Proposition 5.6). More precisely, for each $\mathfrak{D} \in \widetilde{\text{Dec}}^{d=1}(G^m)$, we characterised a subfield $F_m[\mathfrak{D}]$ of $k(\mathfrak{D})_m$ which is isomorphic to $K_m$ and also a number field $F[\mathfrak{D}]$ which is isomorphic to $K$. However, the isomorphism $F_m[\mathfrak{D}] \xrightarrow{\sim} K_m$ is not natural, since it depends on the choice of $\mathfrak{D}$.  \par 
Now we are ready to prove Theorem 1.
\begin{Thm}
Let $m \geq 3$, and let $G$ be a profinite group of $\text{GSC}^{m+9}$-type. Then one can group-theoretically reconstruct a field $F_m(G^m)$ with the action of $G^m$ and the fixed subfield $F(G^m) := F_m(G^m)^{G^m}$ which is a number field, such that: \par
(1) We have a group isomorphism $\text{Gal}(F_m(G^m)/F(G^m)) \xrightarrow{\sim} G_K^{m}$. \par
(2) $F_m(G^m)$ and $F(G^m)$ fit into the following commutative diagram:
$$
\begin{tikzcd}
F_m(G^m) \arrow[r,"\sim"] & K_m \\
F(G^m) \arrow[u,hookrightarrow] \arrow[r,"\sim"] & K \arrow[u,hookrightarrow]
\end{tikzcd}
$$
where the upper horizontal arrow is an $\alpha_m$-equivariant isomorphism and the vertical arrows are natural inclusions. Moreover, the isomorphisms in the diagram are functorial with respective to the isomorphy type of $G$. (More precisely, given an isomorphism $G_0 \xrightarrow{\sim} G$, there is a uniquely determined commutative diagram of fields
$$
\begin{tikzcd}[row sep=scriptsize, column sep=scriptsize]
F_m(G^m) \arrow[rr,"\sim"]  & & F_m(G_0^m)  \\
& K_m \arrow[ul,"\sim"] \arrow[ur,"\sim"] \arrow[from=dd,hookrightarrow, crossing over] & \\
F(G^m) \arrow[uu,hookrightarrow] \arrow[rr,"\sim"] & & F(G_0^m) \arrow[uu,hookrightarrow] \\
& K \arrow[ur,"\sim"] \arrow[ul, "\sim"] \arrow[uu,hookrightarrow , crossing over] &
\end{tikzcd}
$$
where the top arrows are equivariant with respective to the isomorphisms $G_0 \xrightarrow{\sim} G$ and the given isomorphisms $G \xrightarrow{\sim} G_K^{m+9}$ and $G_0 \xrightarrow{\sim} G_K^{m+9}$.)
\end{Thm}
\begin{proof}
First we construct the fields $F_m(G^m)$ and $F(G^m)$. \par 
Let $D$ and $E$ be distinct elements of $\widetilde{\text{Dec}}^{d=1}(G^{m+3})$. Then it follows from Lemma 5.5 that the length $m+3$ standard type subfields
$$
F_{m+3}[D] \subset k(D)_{m+3}~;~ F_{m+3}[E] \subset k(E)_{m+3}
$$
determine uniquely the length $m$ standard type subfields
$$
F_m[\mathfrak{D}] \subset k(\mathfrak{D})_m~;~ F_m[\mathfrak{E}] \subset k(\mathfrak{E})_m
$$
where $\mathfrak{D} := D^m$ and $\mathfrak{E}:=E^m$. \par 
Now consider an isomorphism arising from Definition 5.3 (4) (which exists because both $\text{Gal}(F_{m+3}[D]/F[D])$ and $\text{Gal}(F_{m+3}[E]/F[E])$ are isomorphic to $G^{m+3}$ as in Definition 5.3) 
$$
\sigma_{D,E}^{m+3}: \text{Gal}(F_{m+3}[D]/F[D]) \xrightarrow{\sim} \text{Gal}(F_{m+3}[E]/F[E]),
$$
which induces an isomorphism
$$
\sigma_{\mathfrak{D},\mathfrak{E}}^m : \text{Gal}(F_m[\mathfrak{D}]/F[\mathfrak{D}]) \xrightarrow{\sim} \text{Gal}(F_m[\mathfrak{E}]/F[\mathfrak{E}]),
$$
and hence determines a unique field isomorphism by Theorem 1.2
$$
\tau_{\mathfrak{E},\mathfrak{D}}: F_m[\mathfrak{E}] \xrightarrow{\sim} F_m[\mathfrak{D}]
$$
which restricts to
$$
\tau: F[\mathfrak{E}] \xrightarrow{\sim} F[\mathfrak{D}].
$$
Now consider the ring
$$
\prod_{\mathfrak{D} \in \widetilde{\text{Dec}}^{d=1}(G^m)} F_m[\mathfrak{D}],
$$
and define the subring
$$
F_m(G^m) := \{(\dots,a_{\mathfrak{D}},\dots): \tau_{\mathfrak{D},\mathfrak{E}}(a_{\mathfrak{D}}) = a_{\mathfrak{E}}~\forall \mathfrak{D},\mathfrak{E} \in \widetilde{\text{Dec}}^{d=1}(G^m)\}.
$$
where $\tau_{\mathfrak{D},\mathfrak{E}}$ is the field isomorphism obtained above and $a_{\mathfrak{D}} \in F_m[\mathfrak{D}]$. First, we verify that $F_m(G^m)$ is indeed a field. Let $(\dots,a_{\mathfrak{D}},\dots) \in F_m(G^m)$ be a non-zero element, then $a_{\mathfrak{D}} \neq 0$ for all $\mathfrak{D}$ by the construction of $F_m(G^m)$. On the other hand, $(\dots,a_{\mathfrak{D}},\dots)$ is invertible because each component $a_{\mathfrak{D}}$ is a non-zero element of the field $F_m[\mathfrak{D}]$. Furthermore, one has $\tau_{\mathfrak{D},\mathfrak{E}}(a_{\mathfrak{D}}^{-1}) = a_{\mathfrak{E}}^{-1}$. Hence $F_m(G^m)$ is a field. \par 
Now we define an action of $G^m$ on $\prod_{\mathfrak{D} \in \widetilde{\text{Dec}}^{d=1}(G^m)} F_m[\mathfrak{D}]$ as follows: for each component $F_m[\mathfrak{D}]$, by Definition 5.3, we have an isomorphism $G^m \xrightarrow{\sim}\text{Gal}(F_m[\mathfrak{D}]/F[\mathfrak{D}])$. For each $g \in G^m$, we write $g_{\mathfrak{D}}$ for the image of $g$ via the isomorphism $G^m \xrightarrow{\sim}\text{Gal}(F_m[\mathfrak{D}]/F[\mathfrak{D}])$. We then define the action of $G^m$ on $\prod_{\mathfrak{D} \in \widetilde{\text{Dec}}^{d=1}(G^m)} F_m[\mathfrak{D}]$ by $g(\dots,a_{\mathfrak{D}},\dots) = (\dots,g_{\mathfrak{D}}(a_{\mathfrak{D}}),\dots)$ where $g_{\mathfrak{D}}(a_{\mathfrak{D}})$ is the natural Galois action (c.f. Definition 5.3). \par 
In particular, we obtain an action of $G^m$ on $F_m(G^m)$. Now we check that this action preserves $F_m(G^m)$. Consider the isomorphism
$$
\tau_{\mathfrak{E},\mathfrak{D}}: F_m[\mathfrak{E}] \xrightarrow{\sim} F_m[\mathfrak{D}]
$$
(which is defined earlier), and let $G^m$ act on both sides by automorphism via the isomorphisms arising from Definition 5.3 (4). Then it follows from the definition of $\tau_{\mathfrak{E},\mathfrak{D}}$ and Theorem 1.2, that $\tau_{\mathfrak{E},\mathfrak{D}}$ is $G^m$-equivariant. Hence we can conclude that $G^m$ preserves $F_m(G^m)$ and induces automorphisms of $F_m(G^m)$. Moreover, one verifies immediately that the fixed subfield $F(G^m) := F_m(G^m)^{G^m}$ by this action is a number field, which is isomorphic to $K$. This verifies condition (1). \par 
Now we verify condition (2). Notice that there is an $\alpha_m$-equivariant isomorphism of rings:
$$
\tilde{\tau}:\prod_{\mathfrak{D} \in \widetilde{\text{Dec}}^{d=1}(G^m)} F_m[\mathfrak{D}] \xrightarrow{\sim} \prod_{\widetilde{v} \in \mathscr{P}_{K_m}^{\text{fin},d=1}} K_m[\tilde{v}],
$$
where $K_m[\tilde{v}]$ is defined to be the image of $K_m$ contained in $(K_m)_{\tilde{v}}$, obtained from the component-wise isomorphism $F_m[\mathfrak{D}] \xrightarrow{\sim} K_m[\tilde{v}]$ arising from 
$$
\text{Gal}(F_{m+3}[D]/F[D]) \xrightarrow{\sim} G_K^{m+3}
$$ (c.f. Definition 5.3 and Theorem 1.2) for some length $m+3$-standard type subfields that determines $F_m[\mathfrak{D}]$. Then it is immediate that the image of $F_m(G^m)$ via $\tilde{\tau}$ coincide with the image of $K_m$ in $\prod_{\tilde{v} \in \mathscr{P}_{K_m}^{\text{fin},d=1}} K_m[\tilde{v}]$ via the diagonal embedding. Hence we can conclude that there is an $\alpha_m$-equivariant isomorphism of fields $F_m(G^m)\xrightarrow{\sim} K_m$ and we have the following commutative diagram:
$$
\begin{tikzcd}
F_m(G^m) \arrow[r,"\sim"] & K_m \\
F(G^m) \arrow[u,hookrightarrow] \arrow[r,"\sim"] & K \arrow[u,hookrightarrow].
\end{tikzcd}
$$
Finally, the functoriality (with respect to the isomorphy type of $G$) of the assignment $G \mapsto F_m(G^m)$ follows from Theorem 1.2 together with the fact that the assignment $G \mapsto \widetilde{\text{Dec}}(G^{m+3})$ is functorial with respective to the isomorphy type of $G$ (c.f. Theorem 1.25 in \cite{ST1}).
\end{proof}
\begin{rem}
We have seen why we need to take $m\geq 3$ (c.f. Proposition 5.4, Lemma 5.5 and Proposition 5.6). Now we summaries where the $9$ extra steps are used in Theorem 5.7. \par 
(1) We need to group-theoretically reconstruct the local invariants for each decomposition group, c.f. Theorem 2.6, and also the Kummer containers, this requires 6 extra steps. \par 
(2) We need to apply the main statement of Theorem 1.2 (e.g. we have already used Theorem 1.2 several times in the proof of Theorem 1), this requires $3$ extra steps. 
\end{rem}
\section{Reconstruction of Imaginary Quadratic Fields}
In this section we will review Hoshi's Lemma on NF-monoids in \cite{Ho1} and hence give the proof of Theorem 2. \par 
Before we review Hoshi's construction of NF-monoids, we shall introduce a few notations. Let $K$ be a number field.
\begin{itemize}
\item Let $v \in \mathscr{P}_K^{\text{fin}}$, we write $\mathcal{O}_{(v)}$ for the localisation of the ring of integers $\mathcal{O}_K$ at the prime ideal corresponding to the place $v$.
\item We write $\text{ord}_v: K^{\times} \twoheadrightarrow \mathbb{Z}$ for the valuation map on $K^{\times}$ induced by $v$.
\item We write $\kappa(v)$ for the residue field of $\mathcal{O}_{(v)}$. 
\item We write $\mathcal{O}_{(v)}^{\prec} := \text{ker}(\mathcal{O}_{(v)}^{\times} \twoheadrightarrow \kappa(v)^{\times})$.
\end{itemize}
\begin{defn}[c.f. Definition 2.2 in \cite{Ho1}]
If $K$ is a number field, we call the collection
$$
(K_{\times},\mathcal{O}_K^{\rhd},\mathscr{P}_K^{\text{fin}},\{\mathcal{O}_{(v)}^{\prec}\}_{v \in \mathscr{P}_K^{\text{fin}}})
$$
the \textbf{NF-monoid associated to $K$}.
\end{defn}
\begin{defn}[c.f. Definition 2.3 in \cite{Ho1}]
Let $M$ be a multiplicative monoid, $O^{\rhd} \subset M$ is a submonoid, $S$ is a set and for each $s \in S$, there is a submonoid $O_s^{\prec} \subset M$ of M. We call the collection
$$
\mathcal{M} := (M,O^{\rhd},S,\{O^{\prec}_s\}_{s \in S})
$$
an \textbf{NF-monoid} if there is an NF-monoid associated to some number field $K$ that is isomorphic to $\mathcal{M}$. More precisely, there are isomorphisms of monoids $K_{\times} \xrightarrow{\sim} M$, $\mathcal{O}_K^{\rhd} \xrightarrow{\sim} O^{\rhd}$; a bijection of sets $\mathscr{P}_K^{\text{fin}} \xrightarrow{\sim} S$; and for each $v \in \mathscr{P}_K^{\text{fin}}$ and its image $s$ in $S$, there is an isomorphism of monoids $\mathcal{O}_{(v)}^{\prec} \xrightarrow{\sim} O^{\prec}_s$.
\end{defn}
\begin{lem}[c.f. Theorem 2.9 in \cite{Ho1}]
If $\mathcal{M} := (M,O^{\rhd},S,\{O^{\prec}_s\}_{s \in S})$ is an NF-monoid, then one can define an additive operation $+_{M}$ on $M$ such that $(M,+_{M})$ has a structure of a number field. In particular, if $(K_{\times},\mathcal{O}_K^{\rhd},\mathscr{P}_K^{\text{fin}},\{\mathcal{O}_{(v)}^{\prec}\}_{v \in \mathscr{P}_K^{\text{fin}}})$ is the NF-monoid associated to $K$ that is isomorphic to $\mathcal{M}$, then the isomorphism of monoids
$$
\phi: K_{\times} \xrightarrow{\sim} M
$$
induces an isomorphism of fields
$$
\phi^+: (K_{\times},+_{K_{\times}}) \xrightarrow{\sim} (M,+_M).
$$
\end{lem}
\begin{proof}
See Section 2 in \cite{Ho1}.
\end{proof}
\begin{lem}
Let $K$ be a number field, and let $\alpha:G \xrightarrow{\sim} G_K^6$ be a profinite group of $\text{GSC}^6$-type. Then the followings hold:  \par 
(i) There is a group-theoretic reconstruction of the local cyclotome $\Lambda(D^1)$ of $D \in \widetilde{\text{Dec}}(G^{4})$ starting from $D$. \par 
(ii) There is a group-theoretic reconstruction of the global cyclotome $\Lambda(G^1)$ starting from $G.$ \par 
(iii) There is a group-theoretic reconstruction of the Kummer containers $\mathcal{H}^{\times}(G^1)$ and $\mathcal{H}_{\times}(G^1)$ starting from $G$.
\end{lem}
\begin{proof}
Assertion (i) follows immediately from Theorem 2.6 (ix). Assertion (ii) follows immediately from Definition 3.5. Assertion (iii) follows immediately from Definition 3.10.
\end{proof}
\begin{lem}
We have the following commutative diagram with exact rows
$$
\begin{tikzcd}
1 \arrow[r] & \mathcal{O}_K^{\times} \arrow[r] \arrow[d,hookrightarrow] & K^{\times} \arrow[r] \arrow[d,hookrightarrow] & K^{\times}/O_K^{\times} \arrow[r] \arrow[d,equal] & 1 \\
1 \arrow[r] & \widehat{\mathcal{O}_K^{\times}} \arrow[r] & \mathcal{H}^{\times}(K) \arrow[r] & K^{\times}/O_K^{\times} \arrow[r] & 1.
\end{tikzcd}
$$
In particular, if $\mathcal{O}_K^{\times}$ is finite, i.e. $K$ is $\mathbb{Q}$ or an imaginary quadratic field, then $K^{\times} \hookrightarrow \mathcal{H}^{\times}(K)$ is an isomorphism, and $K_{\times} \hookrightarrow \mathcal{H}_{\times}(K)$ is an isomorphism of monoids.
\end{lem}
\begin{proof}
See Lemma 3.10 in \cite{Ho1}.
\end{proof}
In particular, we can reconstruct the multiplicative group $F^{\times}(G) := \mathcal{H}^{\times}(G)$ and hence the monoid $F_{\times}(G) := \mathcal{H}_{\times}(G)$ by using Lemma 6.4.
\begin{Thm}
There is a group-theoretic reconstruction, starting from $G$, of an NF-monoid
$$
\mathcal{M}(G) := (M,O^{\rhd},S,\{O_s^{\prec}\}_{s \in S})
$$
that is isomorphic to 
$$
(K_{\times},\mathcal{O}^{\rhd}_K,\mathscr{P}_K^{\text{fin}}, \{O_{(v)}^{\prec}\}_{v \in \mathscr{P}_K^{\text{fin}}}).
$$
In particular, by Lemma 6.3, we can construct a field $(M,+_M,\times_M)$ that is isomorphic to $K$.
\end{Thm}
\begin{proof}
We adapt Proposition 5.3 in \cite{Ho1} in our case. Firstly, we set $M := F_{\times}(G)$. By the discussions above Theorem 6.6, $M$ is isomorphic to $K_{\times}$ as monoids. \par 
Then we construct $O^{\rhd}$ as follows: for each $D \in \widetilde{\text{Dec}}(G^4)$ write $\mathcal{O}^{\rhd}(v)$ for the monoid $D^1 \times_{D^{\text{unr}}} \text{Frob}(D)^{\mathbb{Z}_{\geq 0}}$ where $v \in \text{Dec}(G^1)$ is the unique element corresponding to $D$. One has the following isomorphism
$$
\mathcal{O}_{\mathfrak{p}}^{\rhd} \xrightarrow{\sim} \mathcal{O}^{\rhd}(v)
$$
$$
u \pi^n \mapsto \alpha \text{Frob}(D)^n
$$
where $\mathfrak{p} \in \mathscr{P}_K^{\text{fin}}$ is the unique non-archimedean prime of $K$ corresponding to $v$, $u \in \mathcal{O}_{\mathfrak{p}}^{\times}$ is a unit, $\pi$ is a uniformiser in $K_{\mathfrak{p}}$, $\alpha \in \mathcal{O}^{\times}(v)$ and $n \geq 0$ is a non-negative integer. We write 
$$
\mathcal{O}^{\rhd}(G) := \{a \in F_{\times}(G): a \in \mathcal{O}^{\rhd}(v)~\forall~v \in \text{Dec}(G^1)\}.
$$ 
We write $O^{\rhd} := \mathcal{O}^{\rhd}(G)$ and $S := \text{Dec}(G^1)$. \par 
Finally, for any $v \in \text{Dec}(G^1)$, we write
$$
\kappa(v)^{\times} := \mathcal{O}^{\times}(v)^{(p_v')}
$$
which is isomorphic to the multiplicative group of the residue field of $\mathcal{O}_{\mathfrak{p}}$ hence is also isomorphic to the multiplicative group of the residue field of $\mathcal{O}_{(\mathfrak{p})}$ where $\mathfrak{p} \in \mathscr{P}_K^{\text{fin}}$ corresponding to $v$. \par 
We write $\mathcal{O}_{(v)}(G)^{\times} := \text{ker}[F^{\times}(G) \hookrightarrow k^{\times}(v) \twoheadrightarrow k^{\times}(v)/\mathcal{O}^{\times}(v)]$. Notice that $k^{\times}(v) \twoheadrightarrow k^{\times}(v)/\mathcal{O}^{\times}(v)$ coincides with the valuation map on $k^{\times}(v)$. Since $F^{\times}(G)$ is dense in $k^{\times}(v)$, the composite $F^{\times}(G) \hookrightarrow k^{\times}(v) \twoheadrightarrow k^{\times}(v)/\mathcal{O}^{\times}(v)$ is surjective. Hence the composite $F^{\times} \hookrightarrow k^{\times}(v) \twoheadrightarrow k^{\times}(v)/\mathcal{O}^{\times}(v)$ coincide with the $v$-adic valuation map on $F^{\times}(G)$ and $\mathcal{O}_{(\mathfrak{p})}^{\times} \xrightarrow{\sim} \mathcal{O}_{(v)}^{\times}(G)$. We write
$$
\mathcal{O}_{(v)}(G)^{\prec} := \text{ker}(\mathcal{O}_{(v)}^{\times}(G) \twoheadrightarrow \kappa(v)^{\times})
$$
and we write $O_s^{\prec} := \mathcal{O}_{(v)}(G)^{\prec}$. \par 
Therefore, $\mathcal{M}(G) := (F_{\times}(G),\mathcal{O}^{\rhd}(G),\text{Dec}(G^1),\{\mathcal{O}_{(v)}(G)^{\prec}\}_{v \in \text{Dec}(G^1)})$ is an NF-monoid which is isomorphic to $(K_{\times},\mathcal{O}_K^{\rhd},\mathscr{P}_K^{\text{fin}},\{\mathcal{O}_{(\mathfrak{p})}^{\prec}\}_{\mathfrak{p} \in \mathscr{P}_K^{\text{fin}}})$. Hence by Lemma 6.3, we obtain a field $(F_{\times}(G),+_{F_{\times}(G)},\times_{F_{\times}(G)})$ that is isomorphic to $K$.
\end{proof}
\begin{cor}[Theorem 2]
There is a group-theoretic reconstruction of the field $F(G)$ from $G$ such that
$$
F(G) \xrightarrow{\sim} K.
$$
\end{cor}
\begin{proof}
We just set $F(G) := (F_{\times}(G),+_{F_{\times}(G)},\times_{F_{\times}}(G))$, and the result follows immediately from Theorem 6.6.
\end{proof}
\begin{rem}
We need to work with $G \xrightarrow{\sim} G_K^6$ because we can only group-theoretically reconstruct $\mathcal{H}^{\times}(G^1)$ starting from $G \xrightarrow{\sim} G_K^6$ c.f. Definition 3.10.
\end{rem}
\bibliography{ref.bib}
\bibliographystyle{plain}

\end{document}